\documentclass[final, 12pt,reqno]{amsart}
\usepackage{amssymb,amsmath,graphicx,amsfonts,euscript}
\usepackage{color}
\usepackage[pagewise]{lineno}
\usepackage{showkeys}
\usepackage{fancyhdr}
\usepackage[margin=1in]{geometry}
\newtheorem{theorem}{\bf Theorem}[section]

\newtheorem{proposition}[theorem]{\bf Proposition}
\newtheorem{define}[theorem]{\bf Definition}
\newtheorem{remark}{\bf Remark}

\newtheorem{lemma}[theorem]{\bf Lemma}

\newcommand{\beq}{\begin{equation}}
\newcommand{\eeq}{\end{equation}}
\newcommand{\ben}{\begin{eqnarray}}
\newcommand{\een}{\end{eqnarray}}
\newcommand{\beno}{\begin{eqnarray*}}
\newcommand{\eeno}{\end{eqnarray*}}

\numberwithin{equation}{section}
\makeatother
\title[generalized Keller-Segel system]{Suppression of blow up by mixing in generalized  Keller-Segel system with fractional dissipation }

\author[Binbin Shi and  Weike Wang ]{\sc Binbin Shi$^{1}$, Weike Wang$^{2}$}

\address{$^1$ School of Mathematical Sciences , Shanghai Jiao Tong University, Shanghai, 200240,  P.R.China.}
\email{binbinshi@sjtu.edu.cn}

\address{$^2$ School of Mathematical Sciences  and Institute of Natural Science, Shanghai Jiao Tong University, Shanghai, 200240,  P.R.China.}
\email{wkwang@sjtu.edu.cn}

\begin{document}
\bibliographystyle{abbrv}
\maketitle
\begin{center}
\sc Abstract
\end{center}
In this paper, we consider the Cauchy problem for a generalized parabolic-elliptic Keller-Segel equation with fractional dissipation and the additional mixing effect of advection by an incompressible flow. Under suitable mixing condition on the advection, we study well-posedness of solution with large initial data. We establish the global $L^\infty$ estimate of the solution through nonlinear maximum principle, and obtain the global classical solution.

\vskip .3in
\section{Introduction}

We consider the following generalized parabolic-elliptic Keller-Segel system on torus $\mathbb{T}^d$ with fractional dissipation and the additional mixing effect of advection by an incompressible flow
\begin{equation}\label{eq:1.1}
\begin{cases}
\partial_t\rho+u\cdot \nabla \rho+(-\Delta)^{\frac{\alpha}{2}}\rho+\nabla\cdot(\rho B(\rho))=0,\qquad  &t>0, x\in \mathbb{T}^d,\\
\rho(0,x)=\rho_0(x),&x\in\mathbb{T}^d.
\end{cases}
\end{equation}
Here $\rho(t,x)$ is a real value function of $t$ and $x$, $0<\alpha<2$, the $\mathbb{T}^d$ is the periodic box with dimension $d\geq2$. The quantity $\rho$ denotes the density of microorganisms, $u$ is a divergence free vector field which is an ambient flow. The nonlocal operator $(-\Delta)^{\frac{\alpha}{2}}$ is known as the  Laplacian of the order  $\frac{\alpha}{2}$, which is given by
$$
(-\Delta)^{\frac{\alpha}{2}}\phi(x)=\mathcal{F}^{-1}(|\xi|^\alpha \hat{\phi}(\xi))(x),
$$
where
$$
\hat{\phi}(\xi)=\mathcal{F}(\phi(x))=\int_{\mathbb{R}^d}\phi(x)e^{-ix\cdot\xi}dx,
$$
and $\mathcal{F}$ and $\mathcal{F}^{-1}$ are Fourier transformation and its inverse transformation. The linear vector operator $B$ is called attractive kernel, which could be formally represented as
\begin{equation}\label{eq:1.2}
B(\rho)=\nabla((-\Delta)^{-\frac{d+2-\beta}{2}}\rho),
\end{equation}
and it is explicitly expressed by a convolution of a singular kernel $K$
\begin{equation}\label{eq:1.3}
B(\rho)=\nabla K\ast \rho, \quad \nabla K\sim -\frac{x}{|x|^\beta}\quad  2\leq \beta<d+1.
\end{equation}

\vskip .1in
In the absence of the advection, the equation (\ref{eq:1.1}) is the generalized Keller-Segel system with fractional dissipation
\begin{equation}\label{eq:1.4}
\partial_t\rho+(-\Delta)^{\frac{\alpha}{2}}\rho+\nabla\cdot(\rho B(\rho))=0,\quad \rho(0,x)=\rho_0(x),\quad x\in \Omega,
\end{equation}
where $\Omega$ is $\mathbb{R}^d$ or $\mathbb{T}^d$, and the equation (\ref{eq:1.4}) describes many physical processes involving diffusion and interaction of particles (see \cite{ Biler.1999, Brenner.1999}). It is well known that the solution of the equation (\ref{eq:1.4}) is global existence in one dimension, and for multi-dimensions, the solution can blow up in finite time. Specifically, when $\alpha=2, \beta=d$, the equation (\ref{eq:1.4}) is called classical attractive type Keller-Segel system. In one space dimension, the equation admits large data global in time smooth solution (see \cite{Hillen.2004,Koichi.2001} ). In high dimensions, there are global in time smooth solution when the initial data is small, while the solutions  may exhibit finite-time blowup for large data
(see \cite{Blanchet.2006, Corrias.2004,Kiselev.2016,Nagai.1995,Senba.2002}). The regime $0<\alpha<2, \beta=d$, the equation (\ref{eq:1.4}) is a classical Keller-Segel system with fractional dissipation, which was studied by many people and it corresponds to the so-called anomalous diffusion. For $d=1$ and $0<\alpha\leq1$, the solution of equation (\ref{eq:1.4}) is global if $\|\rho_0\|_{L^{\frac{1}{\alpha}}}\leq C(\alpha)$, and the solution of equation (\ref{eq:1.4}) is global if $1<\alpha<2$ (see \cite{Nikolaos.2010}). While $d\geq2$, the solution of equation (\ref{eq:1.4}) would  blow up in finite time with large data (see \cite{Biler.2010, Hopf.2018,Li.2010,Li.2018}). In the case of $0<\alpha<2, \beta\in [2,d+1), d\geq2$, the equation (\ref{eq:1.4}) is called a generalized Keller-Segel system with fractional dissipation, the solution of equation (\ref{eq:1.4}) always blow up in finite time when the initial data is large  (see \cite{Biler.2010, Hopf.2018,Li.2018}).

\vskip .1in
Recently, the chemotactic process coupled with other mechanism has been extensively studied and obtained some interesting phenomenon. For example, Burczak, Belinch\'{o}n (see \cite{Burczak.2017}) and Tello, Winkler (see \cite{Tello.2007}) proved that a logistic source could prevent the singularity of the solution. A more interesting  problem is the chemotactic process taking place in fluid, the agent involved in chemotactic is also advected by the ambient flow. The problem of chemotactic in fluid flow has been studied (see \cite{Francesco.2010, Duan.2010, Liu.2011, Lorz.2010}). For the possible effects resulting from the interaction of chemotactic and fluid transport process, many people get interested in the suppression of blow up in the  chemotactic model by fluid effect. Kiselev, Xu (see \cite{Kiselev.2016}) and Hopf, Rodrigo (see \cite{Hopf.2018}) obtained the global solution of the equation (\ref{eq:1.1}) by the mixing effect of fluid. Bedrossian and He (see \cite{Bedrossian.2017}) showed that the shear flows was dissipation enhancement for the Keller-Segel system. In this paper, we continue to study the mixing effect of fluid to chemotactic model.

\vskip .1in
Mixing was studied by Constantin, Kiselev,  Ryzhik, and Zlato\v{s}
(see \cite{Constantin.2008}) as the fluid effect. In order to describe the mixing effect, Constantin et al. considered the following heat equation with advection
\begin{equation}\label{eq:1.5}
\phi_t^A(t,x)+Au\cdot\nabla \phi^A(t,x)-\Delta\phi^A(t,x)=0, \quad  \phi^A(0,x)=\phi_0(x),
\end{equation}
and they defined the relaxation enhancing flow. Namely, for every $\tau>0, \delta>0$, there exists a positive constant $A_0=A(\tau, \delta)$, such that for any $A\geq A_0$ and any $\phi_0(x)\in L^2$
$$
\|\phi^A(\tau,\cdot)\|_{L^2}\leq \delta \|\phi_0\|_{L^2},
$$
then incompressible flows $u$ is called relaxation enhancing flow. Here $\phi^A(t,x)$ is the solution of (\ref{eq:1.5}), $\overline{\phi}$ is the average of $\phi_0$ and $\overline{\phi}=0$. And they provided a necessary and sufficient condition for the relaxation enhancing flow. Notice that if there is no dissipation term in (\ref{eq:1.5}), the $L^2$ norm conservated, namely $\|\phi^A\|_{L^2}=\|\phi_0\|_{L^2}$. The result in \cite{Constantin.2008} means that combination of mixing and dissipation produces a significantly stronger dissipative effect than dissipation alone. Specifically,  for a fixed time $\tau$, the $\|\phi^A(\tau,\cdot)\|_{\dot{H}^1}$ is large enough in some sense when $A$ is large enough. So the mixing term is enhancing for the dissipation, it can be useful in the model describing a physical situation which involves fast unitary dynamics with dissipation (see \cite{Kiselev.2012, Kiselev.2016}). We will briefly introduce relaxation enhancing flow and weakly mixing (see, Definition \ref{def:2.4}) in section 2.3, the reader can refer to \cite{Constantin.2008} for more details.

\vskip .1in
For the equation (\ref{eq:1.1}), mixing effect is included in chemotactic model, and our main concern is whether mixing can suppress the blowup phenomenon in finite time. When $\alpha=2, \beta=d, d=2,3$, Kiselev and Xu (see \cite{Kiselev.2016}) established the $L^2$ estimate of the solution in the case of weakly mixing, and obtained the global smooth solution by $L^2$-criterion. Namely, the blowup solution of Keller-Segel system be prevented. For $0<\alpha<2, \beta\in [2,d+1), d\geq2$, Hopf and Rodrigo proved that there exists $L^2$ estimate of the solution by relaxation enhancing flow, and also got the global smooth solution if $\alpha>\max\{\beta-\frac{d}{2},1\}$(see \cite{Hopf.2018}, Theorem 4.5). In particularly, for classical Keller-Segel system with fractional dissipation , when $\alpha>\frac{d}{2},d=2,3$, the solution of (\ref{eq:1.1}) was global smooth. For the smaller lower bounds on $\alpha$ and higher dimension $d$, we require the $L^p(p>2)$ estimate of the solution instead of the $L^2$ estimate. Hopf and Rodrigo only considered the case $\alpha=2, \beta=d,$  with $d\geq4$ (see \cite{Hopf.2018}, Theorem 4.6), they got the $L^p(2<p<\infty)$ estimate of the solution by relaxation enhancing flow, and obtained the global smooth solution by $L^p$-criterion.

\vskip .1in
At the same time, Hopf and Rodrigo  thought that the $L^p(p>2)$ estimate of the solution for equation (\ref{eq:1.1})  is hard to achieved in the case of $0<\alpha<2,\beta\in [2,d+1), d\geq2$. The main difficulty from the inequality
\begin{equation}\label{eq:1.8}
\|(-\Delta)^{\frac{1}{2}-\frac{\alpha}{4}}f\|_{L^{p_1}}\leq C\|(-\Delta)^{\frac{\alpha}{4}}(|f|^{\frac{p}{2}})\|^{\frac{2}{p}}_{L^{2}},
\end{equation}
for some $p_1>2$ (see \cite{Hopf.2018}). Certainly the inequality (\ref{eq:1.8}) cannot hold unless $\alpha>\left(\frac{1}{p}+\frac{1}{2}\right)^{-1}>1$, and notice that the $\alpha$ tends to $2$ when $p$ is large enough. So it is  not obvious to extend Hopf and Rodrigo's approach to  the generalized Keller-Segel system with fractional dissipation of any strength $\alpha$ and in any dimension $d\geq2$.

\vskip .1in
In this paper, we consider the generalized Keller-Segel system with fractional dissipation and weakly mixing in the case of any $0<\alpha<2, \beta\in [2,d], d\geq2$. And for convenience, we consider the $\mathbb{T}^d=[-\frac{1}{2},\frac{1}{2})^d$. In order to get $L^p$ estimate of the solution to equation (\ref{eq:1.1}), we introduce a nonlinear maximum principle on $\mathbb{T}^d$ (see Appendix). Due to mixing effect, we obtain the $L^p(p=\infty)$ estimate of the solution through nonlinear maximum principle, then we get the global classical solution by $L^\infty$-criterion. We believe that the range of $\alpha$ and $d$ are more general in our results, comparing with other results in \cite{Hopf.2018, Kiselev.2016}. Due to technical difficulties, we don't consider the case of $d<\beta<d+1$.

\vskip .1in
Let us now state our main result.
\begin{theorem}\label{thm:1.1}
Let $0<\alpha<2,\beta\in [2,d], d\geq2$, for any initial data $\rho_0\geq0, \rho_0\in H^3(\mathbb{T}^d)\cap L^{\infty}(\mathbb{T}^d)$, there exists a smooth incompressible flow $u$, such that the unique solution $\rho(t,x)$ of equation (\ref{eq:1.1}) is global in time, and we have
$$
\rho(t,x)\in C(\mathbb{R}^{+}; H^3(\mathbb{T}^d)).
$$
\end{theorem}

\begin{remark}
The smooth incompressible flow $u$ is weakly mixing (see, Definition \ref{def:2.4}), and the result is still remain true  for the general relaxation enhancing flow (see \cite{Constantin.2008,Hopf.2018}).
\end{remark}


\begin{remark}
The result can seen as a extension of Kiselev et al.(see \cite{Kiselev.2016}) and Hopf et al. (see \cite{Hopf.2018}). I believe that the equation (\ref{eq:1.1}) have same result for sufficiently strong fractional dissipation in the case of $d<\beta<d+1$, but it requires new ideas. We will come back to this topic in the future.
\end{remark}

\vskip .1in
In the following, we briefly state our main ideas of the proof. Firstly, we establish the $L^\infty$-criterion of solution to equation (\ref{eq:1.1}). Namely, we can get the higher order energy estimate of the solution if the $L^\infty$ norm of solution is uniform bound, thus  there is a global classical solution for the equation (\ref{eq:1.1}). Next, we obtain the $L^\infty$ estimate of the solution to equation (\ref{eq:1.1}). According to the $L^1$ norm conservation of the solution and nonlinear maximum principle, we can get the local $L^\infty$ estimate of the solution, it derives that the local $L^2$ estimate of the solution is small by mixing effect. Combining with the local $L^2$ and $L^\infty$ estimate of the solution, we deduce that the local $L^p(p>\frac{d}{a})$ estimate of the solution is controlled by its initial data. Using the nonlinear maximum principle again, the local $L^\infty$ norm is estimated by the initial data. Repeating the above process, we extend the local $L^p$ and $L^\infty$ estimate of the sulution to all time. Thus, we get the uniform $L^\infty$ estimate. In the details of the proof, we will discuss $\beta=d$ and $\beta\in [2,d)$ respectively, due to the properties of $\Delta K$.  When $\beta=d$, $\Delta K$ is not integrable, but the attractive kernel is written as  $B(\rho)=\nabla((-\Delta)^{-1}\rho)$. While $2\leq\beta<d$, the attractive kernel can be expressed by $B(\rho)=\nabla K\ast \rho$, and $\Delta K$ is integrable. So, some different techniques are required to deal with two cases.

\vskip .1in
This paper is organized as follows. In Section 2, we introduce the properties of the nonlocal operator and the functional space. We give the local well-posedness and basic properties for the generalized Keller-Segel system with fractional dissipation and weakly mixing. The mixing effect of the solution is also introduced in this section. In Section 3, we establish the $L^\infty$ estimate of the solution to equation (\ref{eq:1.1}) when $\beta=d$ with $d\geq2$, and we give the proof of Theorem \ref{thm:1.1} by $L^\infty$-criterion. In Section 4, we finish the proof of Theorem \ref{thm:1.1} through the similar method in the case of $\beta\in [2,d), d>2$. Because different properties of $\Delta K$, we introduce some different techniques to complete the proof. In Section 5, we  prove a nonlinear maximum principle on periodic box.

\vskip .1in
Throughout the paper, $C$ stand for universal constants that may change from line to line.

\vskip .3in
\section{Preliminaries}
In what follows, we provide some the auxiliary results and notations.
\subsection{Nonlocal operator}
The fractional Laplacian is nonlocal operator and it has the following kernel representation on $\mathbb{T}^d$ (see \cite{Burczak.2017, Calderon.1954})
\begin{equation}\label{eq:2.1}
(-\Delta)^{\frac{\alpha}{2}}f(x)=C_{\alpha,d}\sum_{k\in \mathbb{Z}^d}P.V.\int_{\mathbb{T}^d}\frac{f(x)-f(y)}{|x-y+k|^{d+\alpha}}dy,
\end{equation}
where
$$
C_{\alpha,d}=\frac{2^\alpha\Gamma(\frac{d+\alpha}{2})}{\pi^{\frac{d}{2}}|\Gamma(-\frac{\alpha}{2})|},
$$
and
$$
\Gamma(z)=\int_0^\infty t^{z-1} e^{-t}dt.
$$
Recall that we denote by
$$
0\leq\lambda_1\leq\lambda_2\leq\cdots \leq\lambda_n\leq \cdots
$$
is the eigenvalue of the operator $-\Delta$ on $\mathbb{T}^d$, then the eigenvalue of operator $(-\Delta)^{\frac{\alpha}{2}}$ is as follows (see \cite{Cusimano.2018})
\begin{equation}\label{eq:2.2}
0\leq\lambda_1^{\frac{\alpha}{2}}\leq\lambda_2^{\frac{\alpha}{2}}\leq\cdots\leq\lambda_n^{\frac{\alpha}{2}}\leq \cdots.
\end{equation}

The following results are two important lemmas.

\begin{lemma}[Positivity Lemma, see \cite{Cordoba.2004, Ju.2005}]\label{lem:2.1}
 Suppose $0\leq \alpha \leq2, \Omega=\mathbb{R}^d, \mathbb{T}^d$ and $f, (-\Delta)^{\frac{\alpha}{2}}f\in L^p$, where $p\geq2$. Then
$$
\frac{2}{p}\int_{\Omega}((-\Delta)^{\frac{\alpha}{4}}|f|^{\frac{p}{2}})^2dx\leq \int_{\Omega}|f|^{p-2}f(-\Delta)^{\frac{\alpha}{2}}f dx.
$$
\end{lemma}
\begin{lemma}\label{lem:2.2}
Suppose $0< \alpha<2, \Omega=\mathbb{R}^d, \mathbb{T}^d$ and $f\in \mathcal{S}(\Omega)$. Then
$$
\int_{\Omega}(-\Delta)^{\frac{\alpha}{2}}f(x)dx=0.
$$
\end{lemma}
\begin{proof}
We only prove $\Omega=\mathbb{T}^d$, according to (\ref{eq:2.1}), one has
$$
\begin{aligned}
&\int_{\mathbb{T}^d}(-\Delta)^{\frac{\alpha}{2}}f(x)dx\\
&=C_{\alpha,d}\sum_{k\in \mathbb{Z}^d}P.V.\int_{\mathbb{T}^d}\int_{\mathbb{T}^d}\frac{f(x)-f(y)}{|x-y+k|^{d+\alpha}}dxdy\\
&=C_{\alpha,d}\sum_{k\in \mathbb{Z}^d}P.V.\left(\int_{\mathbb{T}^d}\int_{\mathbb{T}^d}\frac{f(x)}{|x-y+k|^{d+\alpha}}dxdy
-\int_{\mathbb{T}^d}\int_{\mathbb{T}^d}\frac{f(y)}{|x-y+k|^{d+\alpha}}dxdy \right)\\
&=C_{\alpha,d}\sum_{k\in \mathbb{Z}^d}P.V.\left(\int_{\mathbb{T}^d}\int_{\mathbb{T}^d}\frac{f(x+k)}{|x-y+k|^{d+\alpha}}dxdy
-\int_{\mathbb{T}^d}\int_{\mathbb{T}^d}\frac{f(y-k)}{|x-y+k|^{d+\alpha}}dxdy \right).
\end{aligned}
$$
Due to
$$
\sum_{k\in \mathbb{Z}^d}\int_{\mathbb{T}^d}\int_{\mathbb{T}^d}\frac{f(x+k)}{|x-y+k|^{d+\alpha}}dxdy=
\int_{\mathbb{T}^d}\int_{\mathbb{R}^d}\frac{f(x)}{|x-y|^{d+\alpha}}dxdy,
$$
and
$$
\sum_{k\in \mathbb{Z}^d}\int_{\mathbb{T}^d}\int_{\mathbb{T}^d}\frac{f(y-k)}{|x-y+k|^{d+\alpha}}dxdy
=\int_{\mathbb{T}^d}\int_{\mathbb{R}^d}\frac{f(y)}{|x-y|^{d+\alpha}}dydx,
$$
then
$$
\int_{\mathbb{T}^d}(-\Delta)^{\frac{\alpha}{2}}f(x)dx=0.
$$
We complete the proof.
\end{proof}

\vskip .1in
\subsection{Functional spaces and inequalities} We write $L^p(\mathbb{T}^d)$ for the usual  Lebesgue space
$$
L^p(\mathbb{T}^d)=\left\{f\ measurable\ s.t. \int_{\mathbb{T}^d}|f(x)|^pdx<\infty \right\},
$$
the norm for the $L^p$ space is denoted as $\|\cdot\|_{L^p}$, it means
$$
\|f\|_{L^p}=\left( \int_{\mathbb{T}^d}|f|^pdx\right)^{\frac{1}{p}},
$$
with natural adjustment when $p=\infty$. The homogeneous Sobolev norm  $\|\cdot\|_{\dot{H}^{s}}$,
$$
\|f\|^2_{\dot{H}^{s}}=\|(-\Delta)^{\frac{s}{2}}f\|^2_{L^2}=\sum_{k\in \mathbb{Z}^d\backslash \{0\}}|k|^{2s}|\hat{f}(k)|^2,
$$
and the non-homogeneous Sobolev norm  $\|\cdot\|_{H^{s}}$,
$$
\|f\|^2_{H^{s}}=\|f\|^2_{L^2}+\|f\|^2_{\dot{H}^{s}}.
$$

For some standard inequalities, we can refer to \cite{Evans.2010,Friedman.1969}. The following inequality is a Sobolev embedding  for fractional derivation (see \cite{Tadahiro.2013}).
\begin{lemma}[Homogeneous Sobolev embedding]\label{lem:2.3}
Suppose $0< \frac{\sigma}{d}<\frac{1}{p}<1$ and define $q\in (p,\infty)$ via
$$
\frac{\sigma}{d}=\frac{1}{p}-\frac{1}{q}.
$$
Then for all $f\in C^{\infty}(\mathbb{T}^d)$ with zero mean
$$
\|f\|_{L^q}\leq C\|(-\Delta)^{\frac{\sigma}{2}}f\|_{L^p}.
$$
\end{lemma}

\vskip .1in
\subsection{Mixing  effect}
Given an incompressible vector field $u$ which is Lipschitz in spatial variables, if we defined the trajectories map by (see \cite{Constantin.2008, Kiselev.2016})
$$
\frac{d}{dt} \Phi_t(x)=u(t,\Phi_t(x)), \quad \Phi_0(x)=x.
$$
Then define a unitary operator $U^t$  acting on $L^2(\mathbb{T}^d)$ as follows
$$
U^t f(x)=f(\Phi_t^{-1}(x)).
$$

Next, we give the definition of weakly mixing (see \cite{Kiselev.2016}).
\begin{define}\label{def:2.4}
The incompressible flow $u$ is called weakly mixing if the spectrum of the operator $U\equiv U^t$ is purely continuous.
\end{define}

\begin{remark}
The incompressible flow $u$ is called relaxation enhancing (see \cite{Constantin.2008}) if the operator $U$ has no eigenfunctions in $\dot{H}^1$ other than a constant function. Obviously, the weakly mixing is relaxation enhancing flow.
\end{remark}

\vskip .1in
Let us denote $\omega(t,x)$  is the unique solution of the equation
\begin{equation}\label{eq:2.3}
\partial_t \omega+u\cdot\nabla \omega=0,\quad \omega(0,x)=\rho_0(x),
\end{equation}
there is the following lemma,
\begin{lemma}\label{lem:2.5}
Suppose that $0<\alpha<2$, $u(t,x)$ is a smooth divergence free vector field for each $t\geq0$. Let $\omega(t,x)$ be the solution of (\ref{eq:2.3}). Then for every $t\geq0$, and for every $\rho_0\in \dot{H}^{\frac{\alpha}{2}}$, we have
$$
\|\omega(t,\cdot)\|_{\dot{H}^{\frac{\alpha}{2}}}\leq F(t)\|\rho_0\|_{\dot{H}^{\frac{\alpha}{2}}},
$$
where
$$
F(t)=\exp\left(\int_0^t D(s)ds \right),
$$
and
$$
D(t)\leq C\|(-\Delta)^{\frac{2\alpha+d+2}{4}}u(t,\cdot)\|_{L^2}.
$$
\end{lemma}
\begin{proof}
We can refer to \cite{Hopf.2018,Kiselev.2016}.
\end{proof}

\begin{remark}
 For the examples of relaxation enhancing flow and weakly mixing, we can refer to \cite{Constantin.2008,Fayad.2006,Fayad.2002,Hopf.2018,Kiselev.2016}.
\end{remark}

\vskip .2in
\subsection{Local well-posedness of (\ref{eq:1.1})} We provide a local existence of the solution to (\ref{eq:1.1}) and some basic properties.

\vskip .1in
\begin{theorem}\label{thm:2.6}
Let $0<\alpha\leq2, \beta\in [2,d+1),d\geq2$, $\rho_0\in H^3(\mathbb{T}^d)$ be a non-negative initial data, then there exist time $T=T(\rho_0, \alpha)>0$ and unique non-negative solution $\rho(t,x)$ of (\ref{eq:1.1}), such that
$$
\rho(t,x)\in C([0,T]; H^3(\mathbb{T}^d)),
$$
and
$$
\|\rho(t,\cdot)\|_{L^1}=\|\rho_0\|_{L^1}.
$$
Furthermore, under the restriction $\alpha>1$, the solution is smooth.
\end{theorem}

\begin{proof}
The proof of Theorem \ref{thm:2.6} is standard and it is similar to the one in \cite{Ascasibar.2013, Li.2010}.
\end{proof}

\vskip .3in
\section{Proof of Theorem \ref{thm:1.1} ($\beta=d, d\geq2$)}
\vskip .1in
In this section, we consider the classical Keller-Segel system with fractional dissipation and weakly mixing. We establish the $L^\infty$-criterion, and get the $L^\infty$ estimate of the solution.
\subsection{$L^\infty$-criterion}
We show that to get the global classical solution of (\ref{eq:1.1}), we only need to have certain control of spatial $L^\infty$ norm of the solution.

\begin{proposition}\label{prop:3.1}
Suppose that $0<\alpha<2,\beta=d, d\geq2$, for any initial data $\rho_0\geq0, \rho_0\in H^3(\mathbb{T}^d)\cap L^{\infty}(\mathbb{T}^d)$. Then the following criterion holds: either the location solution to (\ref{eq:1.1}) extends to a global classical solution or there exists $T^*\in (0,\infty)$, such that
$$
\int_0^\tau \|\rho(t,\cdot)\|_{L^\infty}dt \xrightarrow{\tau\rightarrow T^*}\infty.
$$
\end{proposition}

\begin{proof}
We only need to derive a priori bounds on higher order derivatives in terms of $L^\infty$ norm of the solution. Assume $\rho(t,x)$ is the solution of equation (\ref{eq:1.1}), and $\|\rho(t,\cdot)\|_{L^\infty}$ is bound. Let us multiply  both sides of (\ref{eq:1.1}) by $(-\Delta)^3\rho$ and integrate over $\mathbb{T}^d$, we obtain
\begin{equation}\label{eq:3.42}
\begin{aligned}
\frac{1}{2}\frac{d}{dt}\|\rho\|_{\dot{H}^3}^2&+\int_{\mathbb{T}^d}u\cdot \nabla \rho(-\Delta)^3\rho dx\\
&+\int_{\mathbb{T}^d}(-\Delta)^{\frac{\alpha}{2}}\rho(-\Delta)^3\rho dx+\int_{\mathbb{T}^d}\nabla\cdot(\rho B(\rho))(-\Delta)^3\rho dx=0.
\end{aligned}
\end{equation}
We use step-by-step integration and the incompressibility of $u$ to obtain
\begin{equation}\label{eq:3.43}
\left|\int_{\mathbb{T}^d}u\cdot \nabla \rho(-\Delta)^3\rho dx\right|\leq C\|u\|_{C^3}\|\rho\|_{\dot{H}^3}^2.
\end{equation}
And the third term of the left hand side of (\ref{eq:3.42}) is equal to
\begin{equation}\label{eq:3.44}
\int_{\mathbb{T}^d}(-\Delta)^{\frac{\alpha}{2}}\rho(-\Delta)^3\rho dx=\|\rho\|_{\dot{H}^{3+\frac{\alpha}{2}}}^2.
\end{equation}
For the fourth term of the left hand side of (\ref{eq:3.42}), we yield that
\begin{equation}\label{eq:3.45}
\begin{aligned}
&\int_{\mathbb{T}^d}\nabla\cdot(\rho B(\rho))(-\Delta)^3\rho dx=\int_{\mathbb{T}^d}\nabla\rho\cdot(\nabla(-\Delta)^{-1}\rho)(-\Delta)^3\rho dx-\int_{\mathbb{T}^d}\rho^2(-\Delta)^3\rho dx.
\end{aligned}
\end{equation}
According to integrating by parts to the second term of the right hand side of (\ref{eq:3.45}), we  obtain that
$$
\int_{\mathbb{T}^d}\rho^2(-\Delta)^3\rho dx
\sim\sum_{l=0}^3\int_{\mathbb{T}^d}D^l\rho D^{3-l}\rho D^3\rho dx,
$$
where $l=0,1,2.3$ and $D$ denote any partial derivative. By H\"{o}lder inequality, one has
$$
\sum_{l=0,3}\int_{\mathbb{T}^d}D^l\rho D^{3-l}\rho D^3\rho dx\leq 2\|\rho\|_{L^\infty}\|\rho\|_{\dot{H}^{3}}^2,
$$
and
$$
\sum_{l=1}^2\int_{\mathbb{T}^d}D^l\rho D^{3-l}\rho D^3\rho dx\leq
2\|D\rho\|_{L^6}\|D^2\rho\|_{L^3}\|\rho\|_{\dot{H}^{3}}.
$$
By interpolation inequality, then for any $1\leq q'_0\leq\infty$, we get
\begin{equation}\label{eq:3.46}
\|\rho\|_{L^{q'_0}}\leq \|\rho\|^{\frac{1}{q'_0}}_{L^1}\|\rho\|^{1-\frac{1}{q'_0}}_{L^\infty},
\end{equation}
by Gagliardo-Nirenberg inequality, there exist $1\leq q'_1, q'_2\leq\infty$, such that
\begin{equation}\label{eq:1}
\|D \rho\|_{L^6}\leq C\|\rho\|^{1-\theta_1}_{L^{q'_1}}\|\rho\|^{\theta_1}_{\dot{H}^{3}},
\end{equation}
and
\begin{equation}\label{eq:2}
\|D^2 \rho\|_{L^3}\leq C\|\rho\|^{1-\theta_2}_{L^{q'_2}}\|\rho\|^{\theta_2}_{\dot{H}^{3}},
\end{equation}
where
$$
\theta_1=\frac{6-d(1-\frac{6}{q'_1})}{18-d(3-\frac{6}{q'_1})},\quad \theta_2=\frac{12-d(2-\frac{6}{q'_2})}{18-d(3-\frac{6}{q'_2})}.
$$
Due to $\|\rho\|_{L^1}$ conservation and $\|\rho\|_{L^\infty}$ is bound, according to (\ref{eq:3.46}), (\ref{eq:1}) and (\ref{eq:2}), we obtain
$$
\begin{aligned}
\|D\rho\|_{L^6}\|D^2\rho\|_{L^3}\|\rho\|_{\dot{H}^{3}}\leq
C\|\rho\|^{1-\theta_1}_{L^{q'_1}}\|\rho\|^{1-\theta_2}_{L^{q'_2}}
\|\rho\|^{1+\theta_1+\theta_2}_{\dot{H}^{3}}\leq C\|\rho\|^{1+\theta_1+\theta_2}_{\dot{H}^{3}}.
\end{aligned}
$$
Therefore, we have
\begin{equation}\label{eq:3.47}
\int_{\mathbb{T}^d}\rho^2(-\Delta)^3\rho dx\leq C\|\rho\|_{L^\infty}\|\rho\|_{\dot{H}^{3}}^2+C\|\rho\|^{1+\theta_1+\theta_2}_{\dot{H}^{3}}.
\end{equation}
And the first term of the right hand side of (\ref{eq:3.45}), integrating by parts, we get terms that can be estimated similarly to the second term of the right hand side of (\ref{eq:3.45}). The only exceptional terms that appear which have different structure (see \cite{Kiselev.2016}) are
$$
\int_{\mathbb{T}^d}(\partial_{i_1}\partial_{i_2}\partial_{i_3}
\nabla\rho)\cdot(\nabla(-\Delta)^{-1}\rho)\partial_{i_1}\partial_{i_2}\partial_{i_3}
\rho dx,
$$
while these can be reduced to
$$
\int_{\mathbb{T}^d}\rho(\partial_{i_1}\partial_{i_2}\partial_{i_3}
\rho)^2 dx,
$$
 and according to the estimation as before, we get
\begin{equation}\label{eq:3.48}
\int_{\mathbb{T}^d}\nabla\rho\cdot(\nabla(-\Delta)^{-1}\rho)(-\Delta)^3\rho dx\leq C\|\rho\|_{L^\infty}\|\rho\|_{\dot{H}^{3}}^2+C\|\rho\|^{1+\theta_1+\theta_2}_{\dot{H}^{3}}.
\end{equation}
Thus, we deduce  by (\ref{eq:3.47}) and (\ref{eq:3.48}) that
\begin{equation}\label{eq:3.49}
\int_{\mathbb{T}^d}\nabla\cdot(\rho B(\rho))(-\Delta)^3\rho dx\leq C\|\rho\|_{L^\infty}\|\rho\|_{\dot{H}^{3}}^2+C\|\rho\|^{1+\theta_1+\theta_2}_{\dot{H}^{3}}.
\end{equation}
Combing (\ref{eq:3.42}), (\ref{eq:3.43}), (\ref{eq:3.44}) and (\ref{eq:3.49}), we imply that
\begin{equation}\label{eq:3.50}
\frac{d}{dt}\|\rho\|_{\dot{H}^3}^2\leq -2\|\rho\|_{\dot{H}^{3
+\frac{\alpha}{2}}}^2
+C(\|u\|_{C^3}+\|\rho\|_{L^\infty})\|\rho\|_{\dot{H}^{3}}^2
+C\|\rho\|^{1+\theta_1+\theta_2}_{\dot{H}^{3}}.
\end{equation}
By Gagliardo-Nirenberg inequality, for any $1\leq q'_3\leq \infty$, we obtain
\begin{equation}\label{eq:3.51}
\|\rho\|_{\dot{H}^3}\leq C\|\rho\|^{1-\theta}_{L^{q'_3}}\|\rho\|^{\theta}_{\dot{H}^{3
+\frac{\alpha}{2}}},
\end{equation}
where
$$
\theta=\frac{\frac{2d}{q'_3}+6-d}{\frac{2d}{q'_3}+6-d+\alpha}.
$$
We denote
$$
\gamma=\frac{2}{\theta}=\frac{\frac{4d}{q'_3}+12-2d+2\alpha}{\frac{2d}{q'_3}+6-d},
$$
according to (\ref{eq:3.46}) and (\ref{eq:3.51}), we get
$$
\|\rho\|^\gamma_{\dot{H}^3}\leq C\|\rho\|^{(1-\theta)\gamma}_{L^{q'_3}}\|\rho\|^{2}_{\dot{H}^{3
+\frac{\alpha}{2}}}\leq C_4 \|\rho\|^{2}_{\dot{H}^{3
+\frac{\alpha}{2}}},
$$
thus, we have
\begin{equation}\label{eq:3.52}
-\|\rho\|^{2}_{\dot{H}^{3+\frac{\alpha}{2}}}\leq
-C_4^{-1}\|\rho\|^\gamma_{\dot{H}^3}\leq-C\|\rho\|^\gamma_{\dot{H}^3}.
\end{equation}
According to (\ref{eq:3.50}) and (\ref{eq:3.52}), one has
\begin{equation}\label{eq:3.53}
\frac{d}{dt}\|\rho\|_{\dot{H}^3}^2\leq -C\|\rho\|^\gamma_{\dot{H}^3}
+C(\|u\|_{C^3}+\|\rho\|_{L^\infty})\|\rho\|_{\dot{H}^{3}}^2
+C\|\rho\|^{1+\theta_1+\theta_2}_{\dot{H}^{3}}.
\end{equation}
Due to the $\|u\|_{C^3}$, $\|\rho\|_{L^\infty}$ are bound, and we choose $q'_3$, such that
$$
2<1+\theta_1+\theta_2<\gamma,
$$
then we know that the $\|\rho\|_{\dot{H}^3}$ is bound. Because $\|\rho\|_{L^2}$ bound is obvious, by the definition of $\|\rho\|_{H^3}$, we imply that the $\|\rho(t,\cdot)\|_{H^3}$ is bound. Namely, There exists a constant $ C_{H^3}=C(\|\rho\|_{L^\infty}, \|\rho_0\|_{H^3})$, such that
$$
\|\rho(t,\cdot)\|_{H^3}\leq C_{H^3}.
$$
This completes the proof of Proposition \ref{prop:3.1}.
\end{proof}

\begin{remark}
 Particularly, if $\|\rho\|_{L^\infty}$  is bounded only in $[0,T]$, then the $\|\rho(t,\cdot)\|_{H^3}$ is bound in $[0,T]$.
\end{remark}

\subsection{$L^\infty$ estimate of $\rho$}
We establish the $L^\infty$ estimate of the solution to equation (\ref{eq:1.1}). The important technique we use is the nonlinear maximum  principle on periodic box, the details can refer to the Appendix.

\vskip .1in
Let us denote by $\overline{x}_t$ the point such that
$$
\rho(t,\overline{x}_t)=\max_{x\in \mathbb{T}^d}\rho(t,x).
$$
Let us also define
$$
\widetilde{\rho}(t)=\rho(t,\overline{x}_t),
$$
then for any fixed $t\geq0$, using vanishing of a derivation at the point of maximum, we see that
$$
\partial_t \rho(t,\overline{x}_t)=\frac{d}{dt}\widetilde{\rho}(t),\quad
(u\cdot \nabla\rho)(t,\overline{x}_t)=u\cdot \nabla\widetilde{\rho}(t)=0,
$$
and
$$
\left(\nabla\cdot(\rho B(\rho))\right)(t,\overline{x}_t)=-(\widetilde{\rho}(t))^2,
$$
if we denote
$$
(-\Delta)^{\frac{\alpha}{2}}\rho(t,x)\big|_{x=\overline{x}_t}
=(-\Delta)^{\frac{\alpha}{2}}\widetilde{\rho}(t),
$$
we deduce that by (\ref{eq:1.1}) the evolution of $\widetilde{\rho}$ follows

\begin{equation}\label{eq:3.1}
\frac{d}{dt}\widetilde{\rho}+(-\Delta)^{\frac{\alpha}{2}}\widetilde{\rho}
-\widetilde{\rho}^2=0.
\end{equation}
According to the nonlinear maximum principle (see Lemma \ref{lem:5.1}), one has
\begin{equation}\label{eq:3.2}
\widetilde{\rho}(t)\leq C(d,p)\|\rho\|_{L^p},
\end{equation}
if not, then we imply that
\begin{equation}\label{eq:3.3}
(-\Delta)^{\frac{\alpha}{2}}\widetilde{\rho}(t)\geq C(\alpha, d,p)\frac{(\widetilde{\rho}(t))^{1+\frac{p\alpha}{d}}}{\|\rho\|_{L^p}^{\frac{p\alpha}{d}}}.
\end{equation}
Thus, we deduce by (\ref{eq:3.1}) and (\ref{eq:3.3}) that
$$
\frac{d}{dt}\widetilde{\rho}=\widetilde{\rho}^2
-(-\Delta)^{\frac{\alpha}{2}}\widetilde{\rho}\leq \widetilde{\rho}^2
-C(\alpha, d,p)\frac{\widetilde{\rho}^{1+\frac{p\alpha}{d}}}{\|\rho\|_{L^p}^{\frac{p\alpha}{d}}},
$$
so we have
\begin{equation}\label{eq:3.4}
\frac{d}{dt}\widetilde{\rho}\leq \widetilde{\rho}^2
-C(\alpha, d,p)\frac{\widetilde{\rho}^{1+\frac{p\alpha}{d}}}{\|\rho\|_{L^p}^{\frac{p\alpha}{d}}}.
\end{equation}

\begin{remark}
For the case of (\ref{eq:3.2}), we can get the $L^\infty$ estimate directly by $L^p$ estimate. So we will only consider the case of (\ref{eq:3.3}) in the proof.
\end{remark}

\vskip .1in
First, we need the local $L^2$ and $L^\infty$ estimates of the solution.

\begin{lemma}\label{lem:3.1}
Let $0<\alpha<2, \beta=d,d\geq2$, $\rho(t,x)$ is the local solution of equation (\ref{eq:1.1}) with initial data $\rho_0(x)\geq0$. Suppose that $\|\rho_0\|_{L^2}= B_0,  \|\rho_0\|_{L^\infty}\leq C_{\infty}$. Then there exists a time $\tau_1>0$, for any $0\leq t \leq \tau_1$, we have
$$
\|\rho(t,\cdot)-\overline{\rho}\|_{L^2}\leq 2(B_0^2-\overline{\rho}^2)^{\frac{1}{2}}, \quad \|\rho(t,\cdot)\|_{L^\infty}\leq 2C_{\infty},
$$
where
$$
\overline{\rho}=\frac{1}{|\mathbb{T}^d|}\int_{\mathbb{T}^d}\rho_0(x)dx.
$$
\end{lemma}

\begin{proof}
Let us multiply both sides of (\ref{eq:1.1}) by $\rho-\overline{\rho}$ and integrate over $\mathbb{T}^d$, we obtain that
\begin{equation}\label{eq:3.5}
\begin{aligned}
\frac{1}{2}\frac{d}{dt}\|\rho-\overline{\rho}\|_{L^2}^2+&\int_{\mathbb{T}^d}u\cdot\nabla \rho(\rho-\overline{\rho})dx\\ &+\int_{\mathbb{T}^d}(-\Delta)^{\frac{\alpha}{2}}\rho(\rho-\overline{\rho})dx
+\int_{\mathbb{T}^d}\nabla\cdot(\rho B(\rho))(\rho-\overline{\rho})dx=0.
\end{aligned}
\end{equation}
Due to the $u$ is incompressible, we get
\begin{equation}\label{eq:3.6}
\int_{\mathbb{T}^d}u\cdot\nabla \rho(\rho-\overline{\rho})dx=0,
\end{equation}
and  according to  Lemma \ref{lem:2.1} and \ref{lem:2.2}, the third term of the left hand side of (\ref{eq:3.5}) be estimated as
\begin{equation}\label{eq:3.7}
\int_{\mathbb{T}^d}(-\Delta)^{\frac{\alpha}{2}}\rho(\rho-\overline{\rho})dx\geq \|\rho\|_{\dot{H}^{\frac{\alpha}{2}}}^2.
\end{equation}
For the fourth term of the left hand side of (\ref{eq:3.5}), we obtain
$$
\begin{aligned}
\int_{\mathbb{T}^d}\nabla\cdot(\rho B(\rho))(\rho-\overline{\rho})dx
=-\int_{\mathbb{T}^d}\rho^2(\rho-\overline{\rho})dx
+\frac{1}{2}\int_{\mathbb{T}^d}\rho(\rho-\overline{\rho})^2dx,
\end{aligned}
$$
and notice that
$$
\int_{\mathbb{T}^d}\rho^2 (\rho-\overline{\rho})dx=\int_{\mathbb{T}^d}(\rho-\overline{\rho})^3dx
+2\overline{\rho}\int_{\mathbb{T}^d}(\rho-\overline{\rho})^2dx,
$$
so we deduce that
\begin{equation}\label{eq:3.8}
\begin{aligned}
\int_{\mathbb{T}^d}\nabla\cdot(\rho B(\rho))(\rho-\overline{\rho})dx&=\frac{1}{2}\int_{\mathbb{T}^d}\rho (\rho-\overline{\rho})^2dx-\int_{\mathbb{T}^d}(\rho-\overline{\rho})^3dx\\
&\ \ \ \ -2\overline{\rho}\int_{\mathbb{T}^d}(\rho-\overline{\rho})^2dx\\
&\ \ \ \leq\left(\frac{1}{2}\|\rho\|_{L^\infty}+
 \|\rho-\overline{\rho}\|_{L^\infty}+2\overline{\rho}\right)\|\rho-\overline{\rho}\|_{L^2}^2.
\end{aligned}
\end{equation}
Due to the solution $\rho(t,x)$ of equation (\ref{eq:1.1}) is $L^1$ norm conservation,
by nonlinear maximum principle and  (\ref{eq:3.4}), one has
\begin{equation}\label{eq:3.9}
\frac{d}{dt}\widetilde{\rho}\leq \widetilde{\rho}^2
-C(\alpha, d)\frac{\widetilde{\rho}^{1+\frac{\alpha}{d}}}{\|\rho\|_{L^1}^{\frac{\alpha}{d}}}.
\end{equation}
If we define
$$
\tau_0=\min\left\{ \frac{1}{2C_\infty}, T\right\},
$$
where $T$ is lifespan.  Because $\|\rho_0\|_{L^\infty}\leq C_\infty$, by solving the differential inequality in (\ref{eq:3.9}), we imply that for any $0\leq t\leq \tau_0$, one has
\begin{equation}\label{eq:3.10}
\|\rho(t,\cdot)\|_{L^\infty}\leq 2C_\infty.
\end{equation}
According to (\ref{eq:3.8}) and (\ref{eq:3.10}), for any $0\leq t\leq \tau_0$, we deduce that
\begin{equation}\label{eq:3.11}
\begin{aligned}
 \int_{\mathbb{T}^d}\nabla\cdot(\rho B(\rho))(\rho-\overline{\rho})dx\leq 3(C_\infty+\overline{\rho})\|\rho-\overline{\rho}\|_{L^2}^2.
\end{aligned}
\end{equation}
Combing (\ref{eq:3.5}), (\ref{eq:3.6}), (\ref{eq:3.7}) and (\ref{eq:3.11}), we imply that
\begin{equation}\label{eq:3.12}
\frac{d}{dt}\|\rho-\overline{\rho}\|_{L^2}^2
 \leq-2\|\rho\|_{\dot{H}^{\frac{\alpha}{2}}}^2+
 6(C_\infty+\overline{\rho})\|\rho-\overline{\rho}\|_{L^2}^2.
\end{equation}
Due to $\|\rho_0-\overline{\rho}\|_{L^2}^2=B_0^2-\overline{\rho}^2$, by solving the differential inequality in (\ref{eq:3.12}),
we imply that there exists
$$
\tau_1=\min\left\{\tau_0, \frac{\ln 4}{6(C_\infty+\overline{\rho})}\right\},
$$
such that for any $0\leq t\leq\tau_1$,
we have
$$
\|\rho(t,\cdot)-\overline{\rho}\|_{L^2}\leq 2(B_0^2-\overline{\rho}^2)^{\frac{1}{2}}.
$$
According to (\ref{eq:3.10}) and the definition of $\tau_1$, for any $0\leq t\leq\tau_1$, we obtain
$$
\|\rho(t,\cdot)\|_{L^\infty}\leq 2C_\infty.
$$
This completes the proof of Lemma \ref{lem:3.1}.
\end{proof}

\begin{remark}
If there exist $0<\tau'\leq\tau_1$, such that $\int_0^{\tau'}\|\rho(t,\cdot)\|_{\dot{H}^{\frac{\alpha}{2}}}^2dt$ is large enough, then the $\|\rho(\tau',\cdot)-\overline{\rho}\|_{L^2}$ is obviously small. If not, we need an approximate of $\rho(t,x)$. And we also get the local $L^2$ estimate of the solution only dependent on $L^\infty$ estimate of the solution.
\end{remark}

\vskip .1in
Next, we give an  approximation lemma.
\begin{lemma}\label{lem:3.2}
Let $0<\alpha<2, \beta=d, d\geq2$, suppose that the vector field $u(t,x)$ is smooth incompressible flow, and $F(t)\in L^\infty_{loc}[0,\infty)$. Let $\rho(t,x), \omega(t,x)$ be the local solution of equation (\ref{eq:1.1}) and (\ref{eq:2.3}) respectively with $\rho_0\in H^3(\mathbb{T}^d)\cap L^{\infty}(\mathbb{T}^d), \rho_0\geq0$. Then for every $t\in [0,T]$, we have
$$
\begin{aligned}
\frac{d}{dt}\|\rho-\omega\|_{L^2}^2&\leq-\|\rho\|_{\dot{H}^{\frac{\alpha}{2}}}^2+
F^2(t)\|\rho_0\|^2_{\dot{H}^{\frac{\alpha}{2}}}+
\|\rho\|_{L^\infty}\|\rho\|^2_{L^2}\\
 &\qquad +C(\|\nabla\rho\|_{L^4}\|\rho\|_{L^4}+\|\rho\|_{L^\infty}\|\rho\|_{L^2})\|\rho_0\|_{L^2}.
\end{aligned}
$$
\end{lemma}

\begin{proof}
By (\ref{eq:1.1}) and (\ref{eq:2.3}), we obtain the equation
\begin{equation}\label{eq:3.13}
\partial_t(\rho-\omega)+u\cdot \nabla(\rho-\omega)+(-\Delta)^{\frac{\alpha}{2}}\rho+\nabla\cdot(\rho B(\rho))=0.
\end{equation}
Let us multiply both sides of (\ref{eq:3.13}) by $\rho-\omega$ and integrate over $\mathbb{T}^d$, then
\begin{equation}\label{eq:3.14}
\begin{aligned}
\frac{1}{2}\frac{d}{dt}\|\rho-\omega\|_{L^2}^2+&\int_{\mathbb{T}^d}u\cdot\nabla (\rho-\omega)(\rho-\omega)dx\\
&+\int_{\mathbb{T}^d}(-\Delta)^{\frac{\alpha}{2}}\rho(\rho-\omega)dx
+\int_{\mathbb{T}^d}\nabla\cdot(\rho B(\rho))(\rho-\omega)dx=0.
\end{aligned}
\end{equation}
Due to the $u$ is incompressible, we easily get
\begin{equation}\label{eq:3.15}
\int_{\mathbb{T}^d}u\cdot\nabla (\rho-\omega)(\rho-\omega)dx=0.
\end{equation}
And the third term of the left hand side of (\ref{eq:3.14}) be estimated as
$$
\int_{\mathbb{T}^d}(-\Delta)^{\frac{\alpha}{2}}\rho(\rho-\omega)dx
=\int_{\mathbb{T}^d}(-\Delta)^{\frac{\alpha}{2}}\rho\rho dx-\int_{\mathbb{T}^d}(-\Delta)^{\frac{\alpha}{2}}\rho\omega dx,
$$
then we deduce by Lemma \ref{lem:2.1}  and H\"{o}lder inequality that
$$
\int_{\mathbb{T}^d}(-\Delta)^{\frac{\alpha}{2}}\rho\rho=\|\rho\|_{\dot{H}^{\frac{\alpha}{2}}}^2,
$$
and
$$
\int_{\mathbb{T}^d}(-\Delta)^{\frac{\alpha}{2}}\rho\omega dx=\int_{\mathbb{T}^d}(-\Delta)^{\frac{\alpha}{4}}\rho(-\Delta)^{\frac{\alpha}{4}}\omega dx\leq
\|\rho\|_{\dot{H}^{\frac{\alpha}{2}}}\|\omega\|_{\dot{H}^{\frac{\alpha}{2}}},
$$
so we  get
\begin{equation}\label{eq:3.16}
\int_{\mathbb{T}^d}(-\Delta)^{\frac{\alpha}{2}}\rho(\rho-\omega)dx
\geq\|\rho\|_{\dot{H}^{\frac{\alpha}{2}}}^2
-\|\rho\|_{\dot{H}^{\frac{\alpha}{2}}}\|\omega\|_{\dot{H}^{\frac{\alpha}{2}}}.
\end{equation}
The fourth term of the left hand side of (\ref{eq:3.14}) be estimate as
\begin{equation}\label{eq:3.17}
\begin{aligned}
\int_{\mathbb{T}^d}\nabla\cdot(\rho B(\rho))(\rho-\omega)dx&=\int_{\mathbb{T}^d}\nabla\cdot(\rho \nabla(-\Delta)^{-1}\rho)(\rho-\omega)dx\\
&=-\frac{1}{2}\int_{\mathbb{T}^d}\rho^3dx-\int_{\mathbb{T}^d}\nabla\cdot(\rho \nabla(-\Delta)^{-1}\rho)\omega dx.
\end{aligned}
\end{equation}
For the first term of the right hand side of (\ref{eq:3.17}), we get
\begin{equation}\label{eq:3.18}
\int_{\mathbb{T}^d}\rho^3dx\leq \|\rho\|_{L^\infty}\|\rho\|_{L^2}^2,
\end{equation}
and for the second term of the right hand side of (\ref{eq:3.17}), by H\"{o}lder inequality and Poincar\'{e} inequality, we obtain that
\begin{equation}\label{eq:3.19}
\begin{aligned}
\int_{\mathbb{T}^d}\nabla\cdot(\rho \nabla(-\Delta)^{-1}\rho)\omega dx&\leq
\|\nabla\cdot(\rho \nabla(-\Delta)^{-1}\rho)\|_{L^2}\|\omega\|_{L^2}\\
&\leq C(\|\nabla\rho\cdot(\nabla(-\Delta)^{-1}\rho)\|_{L^2}
+\|\rho\|_{L^\infty}\|\rho\|_{L^2})\|\omega\|_{L^2}\\
&\leq C(\|\nabla\rho\|_{L^4}\|\rho\|_{L^4}+\|\rho\|_{L^\infty}\|\rho\|_{L^2})\|\omega\|_{L^2}.
\end{aligned}
\end{equation}
Therefore, we deduce by (\ref{eq:3.17}), (\ref{eq:3.18}) and (\ref{eq:3.19}) that
\begin{equation}\label{eq:3.20}
\int_{\mathbb{T}^d}\nabla\cdot(\rho B(\rho))(\rho-\omega)dx\leq
\frac{1}{2}\|\rho\|_{L^\infty}\|\rho\|^2_{L^2}+
C(\|\nabla\rho\|_{L^4}\|\rho\|_{L^4}+\|\rho\|_{L^\infty}\|\rho\|_{L^2})\|\omega\|_{L^2}.
\end{equation}
Combing (\ref{eq:3.14}), (\ref{eq:3.15}), (\ref{eq:3.16}) and (\ref{eq:3.20}), one has
\begin{equation}\label{eq:3.21}
\begin{aligned}
\frac{1}{2}\frac{d}{dt}\|\rho-\omega\|_{L^2}^2&\leq-\|\rho\|_{\dot{H}^{\frac{\alpha}{2}}}^2+
\|\rho\|_{\dot{H}^{\frac{\alpha}{2}}}\|\omega\|_{\dot{H}^{\frac{\alpha}{2}}}+
\frac{1}{2}\|\rho\|_{L^\infty}\|\rho\|^2_{L^2}\\
 &\qquad +C(\|\nabla\rho\|_{L^4}\|\rho\|_{L^4}+\|\rho\|_{L^\infty}\|\rho\|_{L^2})\|\omega\|_{L^2}.
\end{aligned}
\end{equation}
For the the second of the right hand side of (\ref{eq:3.21}), by Young's inequality yield that
\begin{equation}\label{eq:3.22}
\|\rho\|_{\dot{H}^{\frac{\alpha}{2}}}\|\omega\|_{\dot{H}^{\frac{\alpha}{2}}}
\leq\frac{1}{2}\|\rho\|_{\dot{H}^{\frac{\alpha}{2}}}^2+\frac{1}{2}\|\omega\|_{\dot{H}^{\frac{\alpha}{2}}}^2,
\end{equation}
thus, we deduce by  (\ref{eq:3.21}) and (\ref{eq:3.22}) that
$$
\begin{aligned}
\frac{1}{2}\frac{d}{dt}\|\rho-\omega\|_{L^2}^2&\leq-\frac{1}{2}\|\rho\|_{\dot{H}^{\frac{\alpha}{2}}}^2+
\frac{1}{2}\|\omega\|^2_{\dot{H}^{\frac{\alpha}{2}}}+
\frac{1}{2}\|\rho\|_{L^\infty}\|\rho\|^2_{L^2}\\
 &\qquad +C(\|\nabla\rho\|_{L^4}\|\rho\|_{L^4}+\|\rho\|_{L^\infty}\|\rho\|_{L^2})\|\omega\|_{L^2}.
\end{aligned}
$$
By Lemma \ref{lem:2.5}, we have
$$
\begin{aligned}
\frac{d}{dt}\|\rho-\omega\|_{L^2}^2&\leq-\|\rho\|_{\dot{H}^{\frac{\alpha}{2}}}^2+
F^2(t)\|\rho_0\|^2_{\dot{H}^{\frac{\alpha}{2}}}+
\|\rho\|_{L^\infty}\|\rho\|^2_{L^2}\\
 &\qquad +C(\|\nabla\rho\|_{L^4}\|\rho\|_{L^4}+\|\rho\|_{L^\infty}\|\rho\|_{L^2})\|\rho_0\|_{L^2}.
\end{aligned}
$$
This completes the proof of Lemma \ref{lem:3.2}.
\end{proof}

\vskip .1in
Now, we establish global $L^\infty$ estimate of the solution to equation (\ref{eq:1.1}) in the case of weakly mixing.

\vskip .1in
\begin{proposition}[Global $L^\infty$ estimate]\label{prop:3.3}
Let $0<\alpha<2,\beta=d, d\geq2$, suppose $\rho (t,x)$ is the solution of equation (\ref{eq:1.1}) with initial data $\rho_0\geq0, \rho_0\in H^3(\mathbb{T}^d)\cap L^{\infty}(\mathbb{T}^d)$. Then there exist weakly mixing $u$ and a positive constant $C_{L^\infty}$,  such that
$$
\|\rho(t,\cdot)\|_{L^\infty}\leq C_{L^\infty},\quad  t\in [0,+\infty].
$$
\end{proposition}

\vskip .1in
Before starting the proof of Proposition \ref{prop:3.3}, we need one auxiliary result (see \cite{Constantin.2008,Cusimano.2018,Kiselev.2016}). On $\mathbb{T}^d$, we denote by $e_1, e_2,\cdots,e_n,\cdots$ is orthogonal eigenvectors corresponding to the eigenvalue of $(-\Delta)^{\frac{\alpha}{2}}: 0\leq\lambda_1^{\frac{\alpha}{2}}\leq\lambda_2^{\frac{\alpha}{2}}\leq\cdots\leq \lambda_n^{\frac{\alpha}{2}}\leq\cdots$. Let us denote by $P_N$ the
orthogonal projection on the subspace spanned by the first $N$ eigenvectors $e_1, e_2,\cdots,e_N$ and
$$
S=\{\phi \in L^2\big| \|\phi\|_{L^2}=1\}.
$$
The following lemma is an extension of well-know RAGE theorem (see \cite{Constantin.2008,Cycon.1987, Kiselev.2016}).

\begin{lemma}\label{lemma:3.4}
Let $U$ be a unitary operator with purely continuous spectrum define on $L^2(\mathbb{T}^d)$. Let $K\subset S$ be a compact set. Then for every $N$ and $\sigma>0$, there exists $T_c=T(N, \sigma, K, U)$ such that for all $T\geq T_c$ and every $\phi\in K$, we have
$$
\frac{1}{T}\int_0^T \|P_N U^t \phi\|_{L^2}^2dt\leq\sigma.
$$
\end{lemma}

\begin{remark}
We denote $\chi=\chi(|x|\leq R)$ is cutoff function, if $\chi$ instead of $P_N$, then the RAGE theorem tells us that any state in continuous spectrum space will `` infinitely often leave" the ball of radius $R$. This is indeed what we expect physically.

\end{remark}

\vskip .1in
Let us consider the equation
\begin{equation}\label{eq:3.23}
\begin{cases}
\partial_t\rho^A+Au\cdot \nabla \rho^A+(-\Delta)^{\frac{\alpha}{2}}\rho^A+\nabla\cdot(\rho^A B(\rho^A))=0,\qquad  &t>0, x\in \mathbb{T}^d\\
\rho^A(0,x)=\rho_0(x),&x\in\mathbb{T}^d.
\end{cases}
\end{equation}
Here $A$ is the coupling constant regulating strength of the fluid flow that we assume to be large and $A u$ is weakly mixing.

\vskip .1in
We are ready to give the proof of the Proposition \ref{prop:3.3}.

\begin{proof}[Proof of the Proposition \ref{prop:3.3}]
Due to $\rho_0\in H^3(\mathbb{T}^d)\cap L^{\infty}(\mathbb{T}^d)$, without loss of generality, we can assume that
\begin{equation}\label{eq:3.24}
\|\rho_0\|_{L^2}=B_0,\quad  \|\rho_0-\overline{\rho}\|_{L^p}\leq D,\quad \|\rho_0\|_{L^\infty}\leq C_{\infty},
\end{equation}
where $p>\frac{d}{\alpha}$, $B_0, D, C_{\infty}$ are positive constant, and we denote
$$
B_1=\min \left\{(B_0^2-\overline{\rho}^2)^{\frac{1}{2}},\left( \frac{D}{(2C_\infty+\overline{\rho})^{1-\frac{2}{p}}}\right)^{\frac{p}{2}}\right\}.
$$
Due to $\lambda_n^{\frac{\alpha}{2}}$ the eigenvalues of $(-\Delta)^{\frac{\alpha}{2}}$ on $\mathbb{T}^d$, and
$$
\lambda_n^{\frac{\alpha}{2}}\rightarrow\infty,\quad n\rightarrow\infty,
$$
we choose $N$, such that
\begin{equation}\label{eq:3.25}
\begin{aligned}
\lambda_N^{\frac{\alpha}{2}}\geq \max \left\{\frac{2400}{23}(C_\infty+\overline{\rho}), \left(1-\frac{B_1^2}{B_0^2-\overline{\rho}^2}\right)\frac{1}{\tau_1}+3(C_\infty+\overline{\rho}), \frac{2}{\tau_1}\ln \frac{B_0^2-\overline{\rho}^2}{B_1^2}\right\}.
\end{aligned}
\end{equation}
Define the compact set $K\subset S$ by
\begin{equation}\label{eq:3.26}
K=\{\phi\in S\big|\|\phi\|_{\dot{H}^{\frac{\alpha}{2}}}^{2}\leq \lambda_{N}^{\frac{\alpha}{2}}\}.
\end{equation}
Let $U^t$ is the unitary operator associated with weakly mixing flow $u$ in the Definition \ref{def:2.4}.
 Fix $\sigma=\frac{1}{100}$, then we get $T_c=T_c(N, \sigma, K, U)$, which is the time provide by Lemma \ref{lemma:3.4}. We proceed to impose the first condition on $A_0=A_0(T_c, \rho_0, \tau_1)$. For any $A\geq A_0$, we define $\tau$ as follows
$$
\tau=\frac{T_c}{A}\leq \tau_1.
$$
Due to $\|\rho_0\|_{L^2}=B_0$,
then for the solution $\rho^A(t,x)$ of equation (\ref{eq:3.23}), we deduce by Lemma \ref{lem:3.1} and (\ref{eq:3.24}) that
\begin{equation}\label{eq:3.27}
\|\rho^A(t,\cdot)-\overline{\rho}\|_{L^2}\leq 2(B_0^2-\overline{\rho}^2)^{\frac{1}{2}},\quad \|\rho^A(t,\cdot)\|_{L^\infty}\leq 2C_{\infty},\quad 0\leq t \leq \tau_1.
\end{equation}
Next, we consider the equation
$$
\partial_t \omega^A+Au\cdot\nabla \omega^A=0,\quad \omega^A(0,x)=\rho_0(x),
$$
according to the definition of $U^{t}$, one has
$$
\omega^A(t,x)-\overline{\rho}=U^{At}(\rho_0(x)-\overline{\rho}).
$$
Let $(\rho_0-\overline{\rho})/\|\rho_0-\overline{\rho}\|_{L^2}\in K$, by the Lemma \ref{lemma:3.4} and the definition of $\tau$, we deduce that
\begin{equation}\label{eq:3.28}
\begin{aligned}
 \frac{1}{\tau}\int_{0}^{\tau}\|P_N(\omega^A-\overline{\rho})\|_{L^2}^2dt&= \frac{1}{\tau}\int_{0}^{\tau}\|P_N U^{At}(\rho_0-\overline{\rho})\|_{L^2}^2dt \\
 &=\frac{\|\rho_0-\overline{\rho}\|_{L^2}^2}{\tau}\int_{0}^{\tau}\|P_N U^{At}\frac{(\rho_0-\overline{\rho})}{\|\rho_0-\overline{\rho}\|_{L^2}}\|_{L^2}^2dt\\
 &=\frac{\|\rho_0-\overline{\rho}\|_{L^2}^2}{A\tau}\int_{0}^{\tau}\|P_N U^{At}\frac{(\rho_0-\overline{\rho})}{\|\rho_0-\overline{\rho}\|_{L^2}}\|_{L^2}^2dAt\\
 &=\frac{\|\rho_0-\overline{\rho}\|_{L^2}^2}{T_c}\int_{0}^{T_c}\|P_N U^{s}\frac{(\rho_0-\overline{\rho})}{\|\rho_0-\overline{\rho}\|_{L^2}}\|_{L^2}^2ds\\
 &\leq \sigma \|\rho_0-\overline{\rho}\|_{L^2}^2\leq \frac{1}{100}(B_0^2-\overline{\rho}^2).
\end{aligned}
\end{equation}
Since $(\rho_0-\overline{\rho})/\|\rho_0-\overline{\rho}\|_{L^2}\in K$, by the definition of $K$ in (\ref{eq:3.26}), we have
\begin{equation}\label{eq:3.29}
\|\rho_0\|_{\dot{H}^{\frac{\alpha}{2}}}^{2}\leq \lambda_{N}^{\frac{\alpha}{2}}\|\rho_0-\overline{\rho}\|_{L^2}^2
\leq \lambda_{N}^{\frac{\alpha}{2}} (B_0^2-\overline{\rho}^2).
\end{equation}
For any fixed $p^*\in [1,\infty)$, according to (\ref{eq:3.46}) and (\ref{eq:3.27}), there exists a positive constant $C=C(p^*)$, such that
$$
\|\rho^A(t,\cdot)\|_{L^{p^*}}\leq C,\quad  t\in[0,\tau_1],
$$
and we deduce by (\ref{eq:3.53}) and (\ref{eq:3.27}) that $\|\rho^A(t,\cdot)\|_{H^3}$ is bound for any $t\in [0, \tau_1]$. Namely, there is a positive constant $ C^*_{H^3}$, such that
$$
\|\rho^A(t,\cdot)\|_{H^3}\leq C^*_{H^3}, \quad t\in [0, \tau_1],
$$
 by Gagliardo-Nirenberg inequality, we obtain
\begin{equation}\label{eq:3.41+}
\|\nabla \rho^A\|_{L^4}\leq C\|\rho^A\|^{1-\theta_0}_{L^{p^*}}\|\rho^A\|^{\theta_0}_{\dot{H}^2}\leq C_4,\quad t\in [0, \tau_1],
\end{equation}
where
$$
\theta_0=\frac{\frac{1}{4}-\frac{1}{d}-\frac{1}{p^*}}{\frac{1}{2}-\frac{2}{d}-\frac{1}{p^*}}.
$$
Combing (\ref{eq:3.24}), (\ref{eq:3.27}), (\ref{eq:3.29}), (\ref{eq:3.41+}) and Lemma \ref{lem:3.2}, we deduce that for $0<t\leq \tau_1$,
\begin{equation}\label{eq:3.30}
\begin{aligned}
\frac{d}{dt}\|\rho^A-\omega^A\|_{L^2}^2\leq &\lambda_{N}^{\frac{\alpha}{2}} F(At)^2(B_0^2-\overline{\rho}^2)+8C_\infty B_0^2\\
& +C\left(C_4B_0+C_4C_\infty+C_\infty B_0\right)B_0\\
&\leq\lambda_{N}^{\frac{\alpha}{2}} F(At)^2(B_0^2-\overline{\rho}^2)+C\left(C_4B_0+C_4C_\infty+C_\infty B_0\right)B_0.
\end{aligned}
\end{equation}
Due to $F(t)$ is a locally bounded function, we chose $A_1\geq A_0$, when $A\geq A_1$, such that
$$
\begin{aligned}
&\int_0^{\tau}\lambda_{N}^{\frac{\alpha}{2}} F(At)^2(B_0^2-\overline{\rho}^2)+C\left(C_4B_0+C_4C_\infty+C_\infty B_0\right)B_0 dt\\
&\leq\frac{\lambda_{N}^{\frac{\alpha}{2}} (B_0^2-\overline{\rho}^2)}{A}\int_0^{T_c}F(t)^2dt
+ C\left(C_4B_0+C_4C_\infty+C_\infty B_0\right)B_0\tau\\
&\leq \frac{B_0^2-\overline{\rho}^2}{100}.
\end{aligned}
$$
Therefore, we integrate from $0$ to $t$ in both sides of (\ref{eq:3.30}), where $0\leq t\leq \tau$, we obtain
\begin{equation}\label{eq:3.31}
\|\rho^A(t,\cdot)-\omega^A(t,\cdot)\|_{L^2}^2\leq  \frac{B_0^2-\overline{\rho}^2}{100},
\end{equation}
due to $\|\omega^A(t,\cdot)-\overline{\rho}\|_{L^2}^2=\|\rho_0-\overline{\rho}\|_{L^2}^2=B_0^2-\overline{\rho}^2$, we deduce that
\begin{equation}\label{eq:3.32}
\frac{81}{100}(B_0^2-\overline{\rho}^2)\leq \|\rho^A(t,\cdot)-\overline{\rho}\|_{L^2}^2\leq \frac{121}{100}(B_0^2-\overline{\rho}^2), \quad 0\leq t\leq \tau.
\end{equation}
Furthermore, by the estimates (\ref{eq:3.28}) and (\ref{eq:3.31}), we get
\begin{equation}\label{eq:3.33}
\begin{aligned}
\frac{1}{\tau}\int_{0}^{\tau}\|P_N(\rho^A(t,\cdot)-\overline{\rho})\|_{L^2}^2dt
& \leq\frac{2}{\tau}\int_{0}^{\tau}\|P_N(\omega^A(t,\cdot)-\overline{\rho})\|_{L^2}^2dt\\
&\quad +\frac{2}{\tau}\int_{0}^{\tau}\|P_N(\rho^A(t,\cdot)-\omega^A(t,\cdot))\|_{L^2}^2dt\\
&\leq\frac{B_0^2-\overline{\rho}^2}{25}.
\end{aligned}
\end{equation}
For the $\|\rho^A(t,\cdot)\|_{\dot{H}^{\frac{\alpha}{2}}}^2$, we have
$$
\begin{aligned}
 \|\rho^A(t,\cdot)\|_{\dot{H}^{\frac{\alpha}{2}}}^2&=\|\rho^A(t,\cdot)
 -\overline{\rho}\|_{\dot{H}^{\frac{\alpha}{2}}}^2\\
 &\geq
 \|(I-P_N)(\rho^A(t,\cdot)
 -\overline{\rho})\|_{\dot{H}^{\frac{\alpha}{2}}}^2\\
 &=\|(-\Delta)^{\frac{\alpha}{4}}(I-P_N)(\rho^A(t,\cdot)
 -\overline{\rho})\|_{L^2}^2\\
 &\geq \lambda_N^{\frac{\alpha}{2}}\|(I-P_N)(\rho^A(t,\cdot)
 -\overline{\rho})\|_{L^2}^2,
\end{aligned}
$$
and
$$
\|(I-P_N)(\rho^A(t,\cdot)
 -\overline{\rho})\|_{L^2}^2\geq \frac{1}{2}\|(\rho^A(t,\cdot)
 -\overline{\rho})\|_{L^2}^2-\|P_N(\rho^A(t,\cdot)
 -\overline{\rho})\|_{L^2}^2.
$$
Thus, we deduce by (\ref{eq:3.32}) and (\ref{eq:3.33}) that
\begin{equation}\label{eq:3.34}
\begin{aligned}
\frac{1}{\tau}\int_{0}^{\tau}\|\rho^A(t,\cdot)\|_{\dot{H}^{\frac{\alpha}{2}}}^2dt&
\geq  \frac{1}{\tau}\int_{0}^{\tau}\lambda_N^{\frac{\alpha}{2}}\|(I-P_N)(\rho^A(t,\cdot)
 -\overline{\rho})\|_{L^2}^2dt\\
&\geq \frac{\lambda_N^{\frac{\alpha}{2}}}{2\tau}\int_{0}^{\tau}\|(\rho^A(t,\cdot)
 -\overline{\rho})\|_{L^2}^2dt\\
&\quad -\frac{\lambda_N^{\frac{\alpha}{2}}}{\tau}\int_{0}^{\tau}\|P_N(\rho^A(t,\cdot)
 -\overline{\rho})\|_{L^2}^2dt\\
&\geq \frac{81}{200}\lambda_N^{\frac{\alpha}{2}}(B_0^2-\overline{\rho}^2)
-\frac{1}{25}\lambda_N^{\frac{\alpha}{2}}(B_0^2-\overline{\rho}^2)\\
&\geq\frac{73}{200}\lambda_N^{\frac{\alpha}{2}}(B_0^2-\overline{\rho}^2).
\end{aligned}
\end{equation}
According to (\ref{eq:3.12}), we can obtain
\begin{equation}\label{eq:3.35}
\frac{d}{dt}\|\rho^A-\overline{\rho}\|_{L^2}^2\leq-2\|\rho^A\|_{\dot{H}^{\frac{\alpha}{2}}}^2
+6(C_\infty+\overline{\rho})\|\rho^A-\overline{\rho}\|_{L^2}^2,
\end{equation}
we integrate from $0$ to $\tau$ in both sides of (\ref{eq:3.35}), we get
$$
\|\rho^A(\tau,\cdot)-\overline{\rho}\|_{L^2}^2\leq
-2\int_{0}^{\tau}\|\rho^A\|_{\dot{H}^{\frac{\alpha}{2}}}^2dt
+\int_{0}^{\tau}6(C_\infty+\overline{\rho})\|\rho^A-\overline{\rho}\|_{L^2}^2dt+\|\rho_0-\overline{\rho}\|_{L^2}^2.
$$
Combing (\ref{eq:3.24}), (\ref{eq:3.25}), (\ref{eq:3.27}) and (\ref{eq:3.34}), we have
$$
\begin{aligned}
\|\rho^A(\tau,\cdot)-\overline{\rho}\|_{L^2}^2&\leq (B_0^2-\overline{\rho}^2)-2\tau\left(\frac{1}{\tau}\int_{0}^{\tau}\|\rho^A\|_{\dot{H}^
{\frac{\alpha}{2}}}^2dt\right)+\int_{0}^{\tau}24(C_\infty+\overline{\rho})(B_0^2-\overline{\rho}^2)dt\\
&\leq-\frac{73}{100}\lambda_N^{\frac{\alpha}{2}}(B_0^2-\overline{\rho}^2)\tau
+24(C_\infty+\overline{\rho})(B_0^2-\overline{\rho}^2)\tau+(B_0^2-\overline{\rho}^2)\\
&\leq \left(-\frac{73}{100}\lambda_N^{\frac{\alpha}{2}}+
24(C_\infty+\overline{\rho})\right)(B_0^2-\overline{\rho}^2)\tau+(B_0^2-\overline{\rho}^2)\\
&\leq(1-\frac{1}{2}\lambda_N^{\frac{\alpha}{2}}\tau)(B_0^2-\overline{\rho}^2).
\end{aligned}
$$
We definite
$$
k= \left\lfloor\frac{A \tau_1}{2T_c}\right\rfloor,
$$
where $\lfloor\cdot\rfloor$  is downward rectification. Then there exists a $A_2> A_1$, if
$$
A\geq A_2,
$$
and repeat the above process $k$ times, we have
\begin{equation}\label{eq:3.36}
\|\rho^A(k\tau,\cdot)-\overline{\rho}\|_{L^2}\leq (1-\frac{1}{2}\lambda_N^{\frac{\alpha}{2}}\tau)^{\frac{k}{2}}(B_0^2-\overline{\rho}^2)^{\frac{1}{2}}\leq B_1.
\end{equation}
By (\ref{eq:3.10}), (\ref{eq:3.36}) and interpolation inequality, we deduce that
\begin{equation}\label{eq:3.37}
\|\rho^A(k\tau,\cdot)-\overline{\rho}\|_{L^p}\leq \|\rho^A-\overline{\rho}\|^{\frac{2}{p}}_{L^2}\|\rho^A-\overline{\rho}\|^{1-\frac{2}{p}}_{L^\infty}
\leq D.
\end{equation}
According to Theorem \ref{thm:2.6} and (\ref{eq:3.37}), then there exists a positive constant $C_{L^p}=C(\rho_0)$, such that
\begin{equation}\label{eq:3.38}
\|\rho^A(t,\cdot)\|_{L^p}\leq C_{L^p},\quad 0\leq t\leq k\tau.
\end{equation}
If we denote
$$
\widetilde{\rho}(t)=\rho^A(t,\overline{x}_t)=\max_{x\in \mathbb{T}^d}\rho^A(t,x),
$$
then by nonlinear maximum principle and  (\ref{eq:3.4}), one has
\begin{equation}\label{eq:3.39}
\frac{d}{dt}\widetilde{\rho}\leq \widetilde{\rho}^2
-C(\alpha, d,p)\frac{\widetilde{\rho}^{1+\frac{p\alpha}{d}}}{\|\rho\|_{L^p}^{\frac{p\alpha}{d}}},
\end{equation}
according to (\ref{eq:3.38}) and (\ref{eq:3.39}), for any $0<t<k\tau$, we have
\begin{equation}\label{eq:3.40}
\frac{d}{dt}\widetilde{\rho}\leq \widetilde{\rho}^2
-C_3\widetilde{\rho}^{1+\frac{p\alpha}{d}},
\end{equation}
where $ C_3=C(\alpha,d,p)/(C_{L^p})^{\frac{p\alpha}{d}}$. We set
 $$
M_0=\max\{x|x^2-C_3x^{1+\frac{p\alpha}{d}}=0\},
$$
then we denote
$$
C_{L^\infty}=\max\{M_0, \|\rho_0\|_{L^\infty}\},
$$
due to $\alpha>\frac{d}{p}$, then
$$
1+\frac{p\alpha}{d}>2.
$$
By solving the differential inequality of (\ref{eq:3.40}), we deduce that
\begin{equation}\label{eq:3.41}
\|\rho^A(t,\cdot)\|_{L^\infty}\leq C_{L^\infty},\quad 0\leq t\leq k\tau.
\end{equation}
For the solution $\rho(t,x)$ of equation (\ref{eq:1.1}), by the same argument with above, we deduce that for any $n\in\mathbb{Z}^+$, one has
$$
\|\rho(n\tau,\cdot)-\overline{\rho}\|_{L^p}\leq D.
$$
Then by the similar to (\ref{eq:3.38}) and (\ref{eq:3.41}), for any $t\geq 0$, we have
$$
\|\rho(t,\cdot)\|_{L^\infty}\leq C_{L^\infty}.
$$
This completes the proof of Proposition \ref{prop:3.3}.
\end{proof}

\begin{remark}
Without loss of general, we can assume $C_\infty=C_{L^\infty}$ for the completeness of proof.
\end{remark}

Let us prove the Theorem \ref{thm:1.1} briefly.
\begin{proof}[The proof of Theorem \ref{thm:1.1}]
According to Proposition \ref{prop:3.3}, we can deduce that for the solution $\rho$ of equation (\ref{eq:1.1}), one has
$$
\|\rho(t,\cdot)\|_{L^\infty}\leq C_{L^\infty}\quad 0\leq t<\infty,
$$
then due to $L^\infty$-criterion, we know that the $\|\rho\|_{H^3}$ is uniform bound. Namely, by $\|\rho\|_{L^2}$ estimate of the solution and solving different inequality (\ref{eq:3.53}), we obtain
$$
\|\rho\|_{H^3}\leq C_{H^3}.
$$
By using standard continuation argument, we have
$$
\rho(t,x)\in C(\mathbb{R}^{+}; H^3(\mathbb{T}^d)).
$$
This completes the proof of Theorem \ref{thm:1.1}.
\end{proof}

\begin{remark}
In fact, for any $k\geq2$, $\rho_0\in H^k(\mathbb{T}^d)\cap L^\infty(\mathbb{T}^d)$, we can get
$$
\rho(t,x)\in C(\mathbb{R}^{+}; H^k(\mathbb{T}^d)).
$$
\end{remark}

\vskip .3in
\section{Proof of Theorem \ref{thm:1.1} ($\beta \in [2,d),d>2$)}

In this section, we consider the generalized Keller-Segel system with fractional diffusion and weakly mixing in the case of $\beta \in [2,d),d>2$. Due to the proof is the similar to Theorem \ref{thm:1.1}, so we only deal with different details.

\subsection{$L^\infty$-criterion }
We get the global classical solution of equation (\ref{eq:1.1}) if $L^\infty$ estimate of the solution is global bound.

\begin{proposition}\label{prop:4.1}
Suppose that $0<\alpha<2,\beta\in [2,d), d>2$, for any initial data $\rho_0\geq0, \rho_0\in H^3(\mathbb{T}^d)\cap L^{\infty}(\mathbb{T}^d)$. Then the following criterion holds: either the location solution to (\ref{eq:1.1}) extends to a global classical solution or there exists $T^*\in (0,\infty)$, such that
$$
\int_0^\tau \|\rho(t,\cdot)\|_{L^\infty}dt \xrightarrow{\tau\rightarrow T^*}\infty.
$$
\end{proposition}

\begin{proof}
For the fourth term of (\ref{eq:3.42}), we deduce that
\begin{equation}\label{eq:4.12}
\begin{aligned}
\int_{\mathbb{T}^d}\nabla\cdot(\rho B(\rho))(-\Delta)^3\rho dx&=\int_{\mathbb{T}^d}\nabla\cdot(\rho \nabla K\ast\rho)(-\Delta)^3\rho dx\\
&=\int_{\mathbb{T}^d}\nabla\rho\cdot\nabla K\ast\rho(-\Delta)^3\rho dx+\int_{\mathbb{T}^d}\rho\Delta K \ast \rho(-\Delta)^3\rho dx.
\end{aligned}
\end{equation}
According to integrating by parts to the first term of the right hand side of (\ref{eq:4.12}), we obtain
\begin{equation}\label{eq:4.13}
\int_{\mathbb{T}^d}\nabla\rho\cdot\nabla K\ast\rho(-\Delta)^3\rho dx\sim \sum_{l=0}^{3}\int_{\mathbb{T}^d}D^{l}(\nabla\rho)\cdot D^{3-l}(\nabla K\ast \rho) D^{3}\rho dx.
\end{equation}
When $l=0$, for $2\leq p_1<\infty$, according to (\ref{eq:1.2}), we imply that
$$
\begin{aligned}
\int_{\mathbb{T}^d}\nabla\rho\cdot D^{3}(\nabla K\ast \rho) D^{3}\rho dx
&\leq\|D\rho\|_{L^{p_1}}\|D^{3}(\nabla K\ast \rho)\|_{L^{q}}\|\rho\|_{\dot{H}^3}\\
&=\|D\rho\|_{L^{p_1}}\|D^{\beta-d+2} \rho\|_{L^{q}}\|\rho\|_{\dot{H}^3}\\
&\leq C\|D\rho\|_{L^{p_1}}\|D^{2} \rho\|_{L^{p'_1}}\|\rho\|_{\dot{H}^3},
\end{aligned}
$$
and $l=1$, we deduce that
$$
\begin{aligned}
\int_{\mathbb{T}^d}D(\nabla\rho)\cdot D^{2}(\nabla K\ast \rho) D^{3}\rho dx
&\leq \|D^2\rho\|_{L^{p_1}}\|D^{2}(\nabla K\ast \rho)\|_{L^{q}}\|\rho\|_{\dot{H}^3}\\
&=\|D^2\rho\|_{L^{p_1}}\|D^{\beta-d+1} \rho\|_{L^{q}}\|\rho\|_{\dot{H}^3}\\
&\leq C\|D\rho\|_{L^{p'_1}}\|D^{2} \rho\|_{L^{p_1}}\|\rho\|_{\dot{H}^3},
\end{aligned}
$$
where
$$
\frac{1}{p_1}+\frac{1}{q}=\frac{1}{2},
$$
and
$$
\frac{d-\beta}{d}=\frac{1}{p'_1}-\frac{1}{q}.
$$
By Gagliardo-Nirenberg inequality, for $1\leq q_1,q_2<\infty$, we have
$$
\|D\rho\|_{L^{p_1}}\|D^{2} \rho\|_{L^{p'_1}}\|\rho\|_{\dot{H}^3}\leq
C\|\rho\|^{1-\theta_1}_{L^{q_1}}\|\rho\|^{1-\theta_2}_{L^{q_2}}\|\rho\|^{1+\theta_1+\theta_2}_{\dot{H}^3}
\leq C\|\rho\|^{1+\theta_1+\theta_2}_{\dot{H}^3},
$$
where
$$
\theta_1=\frac{\frac{1}{p_1}-\frac{1}{d}-\frac{1}{q_1}}{\frac{1}{2}-\frac{3}{d}-\frac{1}{q_1}},\quad
\theta_2=\frac{\frac{1}{p'_1}-\frac{2}{d}-\frac{1}{q_2}}{\frac{1}{2}-\frac{3}{d}-\frac{1}{q_2}}.
$$
And for $1\leq q_3,q_4<\infty$, we also have
$$
\|D\rho\|_{L^{p'_1}}\|D^{2} \rho\|_{L^{p_1}}\|\rho\|_{\dot{H}^3}\leq C\|\rho\|^{1-\theta_3}_{L^{q_3}}\|\rho\|^{1-\theta_4}_{L^{q_4}}\|\rho\|^{1+\theta_3+\theta_4}_{\dot{H}^3}
\leq C\|\rho\|^{1+\theta_1+\theta_2}_{\dot{H}^3},
$$
where
$$
\theta_3=\frac{\frac{1}{p'_1}-\frac{1}{d}-\frac{1}{q_3}}{\frac{1}{2}-\frac{3}{d}-\frac{1}{q_3}},\quad
\theta_4=\frac{\frac{1}{p_1}-\frac{2}{d}-\frac{1}{q_4}}{\frac{1}{2}-\frac{3}{d}-\frac{1}{q_4}}.
$$
Due to $\Delta K\sim\frac{1}{|x|^\beta}$ and $\beta\in [2,d)$, we  imply that there exists a constant $C_0=C(\beta)$, such that
\begin{equation}\label{eq:4.4}
\|\Delta K\|_{L^1}\leq C_0.
\end{equation}
Then for $l=2$, we get
$$
\int_{\mathbb{T}^d}D^{2}(\nabla\rho)\cdot D(\nabla K\ast \rho) D^{3}\rho dx\leq
 \|\Delta K \ast \rho\|_{L^\infty}\|\rho\|^{2}_{\dot{H}^3}\leq C_0\|\rho\|_{L^\infty}\|\rho\|^{2}_{\dot{H}^3},
$$
and when $l=3$, we imply
$$
\int_{\mathbb{T}^d}D^{3}(\nabla\rho)\cdot (\nabla K\ast \rho) D^{3}\rho dx
=-\frac{1}{2}\int_{\mathbb{T}^d}(D^{3}\rho)^2 \Delta K \ast \rho dx.
$$
Therefore, we have
\begin{equation}\label{eq:4.14}
\int_{\mathbb{T}^d}\nabla\rho\cdot\nabla K\ast\rho(-\Delta)^3\rho dx\leq C(\|\rho\|^{1+\theta_1+\theta_2}_{\dot{H}^3}+\|\rho\|^{1+\theta_3+\theta_4}_{\dot{H}^3})
+C\|\rho\|_{L^\infty}\|\rho\|^{2}_{\dot{H}^3}.
\end{equation}
For the second term of the right hand side of (\ref{eq:4.12}), we get
\begin{equation}\label{eq:4.15}
\int_{\mathbb{T}^d}\rho\Delta K\ast\rho(-\Delta)^3\rho dx\sim \sum_{l=0}^{3}\int_{\mathbb{T}^d}D^{l}\rho D^{3-l}(\Delta K\ast \rho) D^{3}\rho dx.
\end{equation}
when $l=1,2$, the similar with above, we obtain
$$
\sum_{l=1}^{2}\int_{\mathbb{T}^d}D^{l}\rho D^{3-l}(\Delta K\ast \rho) D^{3}\rho dx
\leq C(\|\rho\|^{1+\theta_1+\theta_2}_{\dot{H}^3}+\|\rho\|^{1+\theta_3+\theta_4}_{\dot{H}^3}).
$$
If we $l=0$, for $2< p_2<\infty$, we deduce that
$$
\begin{aligned}
\int_{\mathbb{T}^d}\rho D^{3}(\Delta K\ast \rho) D^{3}\rho dx &\leq \|\rho\|_{L^{p_2}}
\|D^{\beta-d+3} \rho\|_{L^{q_0}}\|\rho\|_{\dot{H}^3}\\
&\leq C\|\rho\|_{L^{p_2}}
\|D^{3} \rho\|_{L^{p_3}}\|\rho\|_{\dot{H}^3}\\
&\leq C\|\rho\|_{L^{p_2}}
\|D^{3} \rho\|_{L^{2}}\|\rho\|_{\dot{H}^3}\\
&\leq C\|\rho\|^2_{\dot{H}^3},
\end{aligned}
$$
where
$$
\frac{1}{p_2}+\frac{1}{q_0}=\frac{1}{2},
$$
and
$$
\frac{d-\beta}{d}=\frac{1}{p_3}-\frac{1}{q_0}, \quad  1<p_3<2.
$$
And when $l=3$, we imply that
$$
\int_{\mathbb{T}^d}D^{3}\rho(\Delta K\ast \rho) D^{3}\rho dx\leq \|\Delta K\ast \rho\|_{L^\infty}\|\rho\|^2_{\dot{H}^3}\leq C\|\rho\|^2_{\dot{H}^3},
$$
so we have
\begin{equation}\label{eq:4.16}
\int_{\mathbb{T}^d}\rho\Delta K\ast\rho(-\Delta)^3\rho dx\leq C(\|\rho\|^{1+\theta_1+\theta_2}_{\dot{H}^3}+\|\rho\|^{1+\theta_3+\theta_4}_{\dot{H}^3})
+C\|\rho\|^{2}_{\dot{H}^3}.
\end{equation}
Thus, we deduce by (\ref{eq:4.14}) and (\ref{eq:4.16}) that
\begin{equation}\label{eq:4.17}
\int_{\mathbb{T}^d}\nabla\cdot(\rho B(\rho))(-\Delta)^3\rho dx\leq C(\|\rho\|^{1+\theta_1+\theta_2}_{\dot{H}^3}+\|\rho\|^{1+\theta_3+\theta_4}_{\dot{H}^3})
+C\|\rho\|^{2}_{\dot{H}^3}.
\end{equation}
Combing (\ref{eq:3.42}), (\ref{eq:3.43}), (\ref{eq:3.44}) and (\ref{eq:4.17}), we imply that
\begin{equation}\label{eq:4.18}
\frac{d}{dt}\|\rho\|_{\dot{H}^3}^2\leq -2\|\rho\|_{\dot{H}^{3
+\frac{\alpha}{2}}}^2
+C(\|u\|_{C^3}+1)\|\rho\|_{\dot{H}^{3}}^2
+C(\|\rho\|^{1+\theta_1+\theta_2}_{\dot{H}^3}+\|\rho\|^{1+\theta_3+\theta_4}_{\dot{H}^3}).
\end{equation}
According to (\ref{eq:3.52}) and  for $1\leq p_0<\infty$, one has
\begin{equation}\label{eq:4.19}
-\|\rho\|^{2}_{\dot{H}^{3+\frac{\alpha}{2}}}\leq
-C_4^{-1}\|\rho\|^\gamma_{\dot{H}^3}\leq-C\|\rho\|^\gamma_{\dot{H}^3}.
\end{equation}
where
$$
\gamma=\frac{\frac{4d}{p_0}+12-2d+2\alpha}{\frac{2d}{p_0}+6-d},\quad 1\leq p_0<\infty.
$$
By (\ref{eq:4.18}) and (\ref{eq:4.19}), we have
\begin{equation}\label{eq:4.20}
\frac{d}{dt}\|\rho\|_{\dot{H}^3}^2\leq -C\|\rho\|^\gamma_{\dot{H}^3}
+C(\|u\|_{C^3}+1)\|\rho\|_{\dot{H}^{3}}^2
+C(\|\rho\|^{1+\theta_1+\theta_2}_{\dot{H}^3}+\|\rho\|^{1+\theta_3+\theta_4}_{\dot{H}^3}).
\end{equation}
We can choose $p_0$, such that
$$
\gamma> \max\{2,1+\theta_1+\theta_2,1+\theta_3+\theta_4\}.
$$
By the differential inequality (\ref{eq:4.20}), then the conclusion can easily be deduced. This completes the proof of Proposition \ref{prop:4.1}.
\end{proof}


\vskip .1in
\subsection{$L^\infty$ estimate of $\rho$}
We obtain the $L^\infty$ estimate of the solution by weakly mixing. The same idea is from Section 3.

\vskip .1in
 Let us denote by $\overline{x}_t$ the point such that
$$
\widetilde{\rho}(t)=\rho(t,\overline{x}_t)=\max_{x\in \mathbb{T}^d}\rho(t,x),
$$
then for a fixed $t\geq0$, for a derivation at the point of maximum, we see that
\begin{equation}\label{eq:4.1}
\left(\nabla\cdot(\rho B(\rho))\right)(t,\overline{x}_t)=\nabla \rho\cdot \nabla K\ast\rho(t,\overline{x}_t)+\rho\Delta K\ast\rho(t,\overline{x}_t).
\end{equation}
For the first term of the right hand side in (\ref{eq:4.1}), one has
\begin{equation}\label{eq:4.2}
\nabla \rho\cdot \nabla K\ast\rho(t,\overline{x}_t)=0,
\end{equation}
and for the second term of the right hand side in (\ref{eq:4.1}), by Young's inequality, we have
\begin{equation}\label{eq:4.3}
\begin{aligned}
\rho\Delta K\ast\rho(t,\overline{x}_t)&\leq\|\rho\Delta K\ast\rho\|_{L^\infty} \\
&\leq \|\rho\|_{L^\infty}\|\Delta K\ast\rho\|_{L^\infty}\\
&\leq\|\rho\|_{L^\infty}^2\|\Delta K\|_{L^1}=\widetilde{\rho}^2\|\Delta K\|_{L^1}.
\end{aligned}
\end{equation}
Thus, combing (\ref{eq:4.4}), (\ref{eq:1.1}), (\ref{eq:4.1}), (\ref{eq:4.2}) and (\ref{eq:4.3}), we implies that the evolution of $\widetilde{\rho}$ follows
\begin{equation}\label{eq:4.5}
\frac{d}{dt}\widetilde{\rho}+(-\Delta)^{\frac{\alpha}{2}}\widetilde{\rho}
-C_0\widetilde{\rho}^2\leq0.
\end{equation}
According to the nonlinear maximum principle, one has
$$
\widetilde{\rho}(t)\leq C(d,p)\|\rho\|_{L^p},
$$
if not, we imply that
\begin{equation}\label{eq:4.6}
\frac{d}{dt}\widetilde{\rho}\leq C_0\widetilde{\rho}^2
-C(\alpha, d,p)\frac{\widetilde{\rho}^{1+\frac{p\alpha}{d}}}{\|\rho\|_{L^p}^{\frac{p\alpha}{d}}}.
\end{equation}

\vskip .1in
We give the local $L^2$ and $L^\infty$ estimates of the solution.

\vskip .1in
\begin{lemma}\label{lem:4.1}
Let $0<\alpha<2, \beta\in[2,d),d>2$, $\rho(t,x)$ is the local solution of equation (\ref{eq:1.1}) with initial data $\rho_0(x)\geq0$. Suppose that $\|\rho_0\|_{L^2}=B_0,  \|\rho_0\|_{L^\infty}\leq C_{\infty}$. Then there exists a time $\tau_1>0$, for any $0\leq t \leq \tau_1$, we have
$$
\|\rho(t,\cdot)-\overline{\rho}\|_{L^2}\leq 2(B_0^2-\overline{\rho}^2)^{\frac{1}{2}}, \quad \|\rho(t,\cdot)\|_{L^\infty}\leq 2C_{\infty},
$$
\end{lemma}

\begin{proof}
Due to $\|\rho_0\|_{L^\infty}\leq C_\infty$, we define
$$
\tau_0=\min \left\{  \frac{1}{2C_0C_{\infty}},T\right\},
$$

Then for the solution $\rho(t,x)$ of equation (\ref{eq:1.1}), due to $L^1$ conservation,  according to $L^1$ conservation of the solution and (\ref{eq:4.6}), we imply that
\begin{equation}\label{eq:4.7}
\|\rho(t,\cdot)\|_{L^\infty}\leq 2C_\infty, \quad 0\leq t\leq \tau_0.
\end{equation}
The fourth term of the left hand side of (\ref{eq:3.5}) can be estimated as
\begin{equation}\label{eq:4.8}
\begin{aligned}
\int_{\mathbb{T}^d} \nabla(\rho B(\rho))(\rho-\overline{\rho})dx&=\int_{\mathbb{T}^d} \nabla(\rho \nabla K\ast \rho)(\rho-\overline{\rho})dx\\
&=\frac{1}{2}\int_{\mathbb{T}^d}(\rho-\overline{\rho})^2 \Delta K\ast \rho dx+\overline{\rho}\int_{\mathbb{T}^d}(\rho-\overline{\rho}) \Delta K\ast \rho dx\\
&\leq \frac{1}{2}\|\Delta K\ast \rho\|_{L^\infty}\|\rho-\overline{\rho}\|_{L^2}^2+\overline{\rho}\|\Delta K\ast \rho\|_{L^\infty}\|\rho-\overline{\rho}\|_{L^\infty}\\
&\leq \frac{C_0}{2}\|\rho\|_{L^\infty}\|\rho-\overline{\rho}\|_{L^2}^2
+C_0\overline{\rho}(\|\rho\|_{L^\infty}+\overline{\rho}).
\end{aligned}
\end{equation}
Combing (\ref{eq:3.5}), (\ref{eq:3.6}), (\ref{eq:3.7}) and (\ref{eq:4.8}), we deduce that
\begin{equation}\label{eq:4.9}
\frac{d}{dt}\|\rho-\overline{\rho}\|_{L^2}^2
 \leq-2\|\rho\|_{\dot{H}^{\frac{\alpha}{2}}}^2+C_0\|\rho\|_{L^\infty}\|\rho-\overline{\rho}\|_{L^2}^2
+2C_0\overline{\rho}(\|\rho\|_{L^\infty}+\overline{\rho}).
\end{equation}
According to (\ref{eq:4.7}) and (\ref{eq:4.9}), we imply that there exists
$$
\tau_1=\min\left\{\tau_0, \frac{ 1}{2C_0C_\infty}\ln \frac{4C_\infty(B_0^2-\overline{\rho}^2)-\overline{\rho}(2C_\infty+\overline{\rho})}
{C_\infty(B_0^2-\overline{\rho}^2)-\overline{\rho}(2C_\infty+\overline{\rho})} \right\},
$$
such that for any $0<t\leq \tau_1$, one has
$$
\|\rho(t,\cdot)-\overline{\rho}\|_{L^2}\leq 2(B_0^2-\overline{\rho}^2)^{\frac{1}{2}}.
$$
According to (\ref{eq:4.7}) and the definition of $\tau_1$, for any $0 \leq t \leq \tau_1$, we obtain
$$
\|\rho(t,\cdot)\|_{L^\infty}\leq 2C_\infty, \quad 0\leq t\leq \tau_1.
$$
This completes the proof of Lemma \ref{lem:4.1}.
\end{proof}
\vskip .1in
Next, We give an  approximation lemma.
\begin{lemma}\label{lem:4.2}
Let $0<\alpha<2, \beta\in [2,d), d>2$, suppose that the vector field $u(t,x)$ is smooth incompressible flow, and $F(t)\in L^\infty_{loc}[0,\infty)$. Let $\rho(t,x), \omega(t,x)$ be the local solution of equation (\ref{eq:1.1}) and (\ref{eq:2.3}) respectively with $\rho_0\in H^3(\mathbb{T}^d)\cap L^{\infty}(\mathbb{T}^d), \rho_0\geq0$. Then for every $t\in [0,T]$, we have
$$
\begin{aligned}
\frac{d}{dt}\|\rho-\omega\|_{L^2}^2
&\leq -\|\rho\|_{\dot{H}^{\frac{\alpha}{2}}}^2+F(t)^2\|\rho_0\|_{\dot{H}^{\frac{\alpha}{2}}}^2
+C\|\rho\|_{L^\infty}\|\rho\|_{L^2}^2\\
&\quad+C(\|\nabla \rho\|_{L^4}\|\rho\|_{L^4}+\|\rho\|_{L^\infty}\|\rho\|_{L^2})\|\rho_0\|_{L^2}.
\end{aligned}
$$
\end{lemma}
\begin{proof}
The fourth term of the left hand side of (\ref{eq:3.14}) can be estimate as
\begin{equation}\label{eq:4.10}
\begin{aligned}
&\int_{\mathbb{T}^d}\nabla\cdot(\rho B(\rho))(\rho-\omega)dx=\int_{\mathbb{T}^d}\nabla\cdot(\rho \nabla K\ast \rho)(\rho-\omega)dx\\
&=\int_{\mathbb{T}^d}\nabla\cdot(\rho \nabla K\ast \rho)\rho dx-\int_{\mathbb{T}^d}\nabla\cdot(\rho \nabla K\ast \rho)\omega dx\\
&=\frac{1}{2}\int_{\mathbb{T}^d}\rho^2 \Delta K\ast \rho dx-\int_{\mathbb{T}^d}\nabla\cdot(\rho\nabla K\ast \rho )\omega dx\\
&\leq \frac{1}{2}\|\Delta K\ast \rho\|_{L^\infty}\|\rho\|_{L^2}^2+\|\nabla\cdot(\rho\nabla K\ast \rho )\|_{L^2}\|\omega\|_{L^2}\\
&\leq \frac{1}{2}\|\Delta K\|_{L^1}\|\rho\|_{L^\infty}\|\rho\|_{L^2}^2
+(\|\nabla \rho\|_{L^4}\|\nabla K\ast \rho\|_{L^4}+\|\rho\|_{L^\infty}\|\Delta K\ast \rho\|_{L^2})\|\omega\|_{L^2}\\
&\leq C\|\rho\|_{L^\infty}\|\rho\|_{L^2}^2+C(\|\nabla \rho\|_{L^4}\|\rho\|_{L^4}+\|\rho\|_{L^\infty}\|\rho\|_{L^2})\|\omega\|_{L^2},
\end{aligned}
\end{equation}
where $\|\nabla K\|_{L^1}, \|\Delta K\|_{L^1}$ is bound since $\beta\in[2,d)$. By Young's inequality and (\ref{eq:4.10}), one has
\begin{equation}\label{eq:4.11}
\begin{aligned}
\int_{\mathbb{T}^d}\nabla\cdot(\rho B(\rho))(\rho-\omega)dx&\leq +C(\|\nabla \rho\|_{L^4}\|\rho\|_{L^4}\|\rho\|_{L^\infty}\|\rho\|_{L^2})\|\omega\|_{L^2}\\
&\quad +C\|\rho\|_{L^\infty}\|\rho\|_{L^2}^2.
\end{aligned}
\end{equation}
Combing (\ref{eq:3.14}), (\ref{eq:3.15}), (\ref{eq:3.16}), (\ref{eq:3.22}), (\ref{eq:4.11}) and Lemma \ref{lem:2.5}, we have
$$
\begin{aligned}
\frac{d}{dt}\|\rho-\omega\|_{L^2}^2
&\leq -\|\rho\|_{\dot{H}^{\frac{\alpha}{2}}}^2+F(t)^2\|\rho_0\|_{\dot{H}^{\frac{\alpha}{2}}}^2
+C\|\rho\|_{L^\infty}\|\rho\|_{L^2}^2\\
&\quad+C(\|\nabla \rho\|_{L^4}\|\rho\|_{L^4}+\|\rho\|_{L^\infty}\|\rho\|_{L^2})\|\rho_0\|_{L^2}.
\end{aligned}
$$
This completes the proof of Lemma \ref{lem:4.2}.
\end{proof}

Next, we establish the global $L^\infty$ estimate of the solution.

\begin{proposition}[Global $L^\infty$ estimate]\label{prop1}

Let $0<\alpha<2,\beta\in [2,d), d>2$, suppose $\rho (t,x)$ is the solution of equation (\ref{eq:1.1}) with initial data $\rho_0\geq0, \rho_0\in H^3(\mathbb{T}^d)\cap L^{\infty}(\mathbb{T}^d)$. Then there exist weakly mixing $u$ and a positive constant $C_{L^\infty}$,  such that
$$
\|\rho(t,\cdot)\|_{L^\infty}\leq C_{L^\infty},\quad  t\in [0,+\infty].
$$
\end{proposition}
\begin{proof}
According to the proof of Proposition \ref{prop:3.3}.
\end{proof}

\begin{remark}
According to Proposition \ref{prop:4.1} and \ref{prop1}, we can finish the proof of Theorem \ref{thm:1.1} when $\beta\in [2,d), d>2$.
\end{remark}

\vskip .3in
\section{Appendix: nonlinear maximum principle}
In this section, we recall nonlinear maximum principle on $\mathbb{T}^d$, the main idea of proof come from \cite{Burczak.2017,Constantin.2012, Rafael.2016}.

\begin{lemma}\label{lem:5.1}
Let $f\in \mathcal{S}(\mathbb{T}^d)$ and denote by $\overline{x}$ the point such that
$$
f(\overline{x})=\max_{x\in \mathbb{T}^d}f(x),
$$
and $f(\overline{x})>0$, then we have the following
$$
(-\Delta)^{\frac{\alpha}{2}}f(\overline{x})\geq C(\alpha,d,p)\frac{f(\overline{x})^{1+\frac{p\alpha}{d}}}{\|f\|_{L^p}^{\frac{p\alpha}{d}}},
$$
or
$$
f(\overline{x})\leq C(d,p)\|f\|_{L^p}.
$$
\end{lemma}

\begin{proof}
We take $R>0$ a positive number and defined
$$
N_1(R)=\big\{\lambda\in B(0,R)\big| f(\overline{x})-f(\overline{x}-\lambda)> \frac{f(\overline{x})}{2}\big\}.
$$
We define
$$
M=\min_{y\in \partial \mathbb{T}^d }|\overline{x}-y|,
$$
without loss of generality, we assume that $M\geq\frac{1}{4}$. If
\begin{equation}\label{eq:5.1}
R\leq M,
\end{equation}
then, we imply
$$
B(0,R)\subset \mathbb{T}^d.
$$
If we denote
$$
N_2(R)=B(0,R)-N_1(R),
$$
then
$$
N_2(R)=\big\{\lambda_N\in B(0,R)\big| f(\overline{x})-f(\overline{x}-\lambda)\leq \frac{f(\overline{x})}{2}\big\},
$$
and
$$
\|f\|_{L^p}^p\geq\int_{\mathbb{T}^d}|f(\overline{x}-\lambda)|^pd\lambda
\geq\int_{N_2(R)}|f(\overline{x}-\lambda)|^pd\lambda\geq \left(\frac{|f(\overline{x})|}{2}\right)^p|N_2(R)|,
$$
thus, we imply that
\begin{equation}\label{eq:5.2}
|N_2(R)|\leq \left(\frac{2\|f\|_{L^p}}{f(\overline{x})}\right)^p.
\end{equation}
According to the definition of (\ref{eq:2.1}), we have
$$
\begin{aligned}
(-\Delta)^{\frac{\alpha}{2}}f(\overline{x})&\geq C_{\alpha,d}P.V.\int_{\mathbb{T}^d}\frac{f(\overline{x})-f(\overline{x}-\lambda))}{|\lambda|^{d+\alpha}}d\lambda\\
&\geq C_{\alpha,d}P.V.\int_{N_1(R)}\frac{f(\overline{x})-f(\overline{x}-\lambda))}{|\lambda|^{d+\alpha}}d\lambda\\
&\geq C_{\alpha,d}\frac{f(\overline{x})}{2}\frac{1}{R^{d+\alpha}}|N_1(R)|.
\end{aligned}
$$
By $(\ref{eq:5.2})$, the definition of $N_1(R)$ and $N_2(R)$, we imply that
$$
|N_1(R)|=|B(0,R)|-|N_2(R)|\geq \omega_d R^d-\left(\frac{2\|f\|_{L^p}}{f(\overline{x})}\right)^p,
$$
where $\omega_d$ is the volume pre sphere. Thus, we obtain
\begin{equation}\label{eq:5.3}
(-\Delta)^{\frac{\alpha}{2}}f(\overline{x})
\geq C_{\alpha,d}\frac{f(\overline{x})}{2R^{d+\alpha}}\left(\omega_d R^d-\left(\frac{2\|f\|_{L^p}}{f(\overline{x})}\right)^p\right).
\end{equation}
We take $R$ such that
$$
\omega_d R^d=2\left(\frac{2\|f\|_{L^p}}{f(\overline{x})}\right)^p,
$$
thus
\begin{equation}\label{eq:5.4}
R=\left(\frac{2}{\omega_d}\left(\frac{2\|f\|_{L^p}}{f(\overline{x})} \right)^p\right)^{\frac{1}{d}}
=\left(\frac{2}{\omega_d}\right)^{\frac{1}{d}}\left(\frac{2\|f\|_{L^p}}{f(\overline{x})} \right)^{\frac{p}{d}}.
\end{equation}
By (\ref{eq:5.3}) and (\ref{eq:5.4}), we have
$$
\begin{aligned}
(-\Delta)^{\frac{\alpha}{2}}f(\overline{x})
&\geq C_{\alpha,d}\frac{f(\overline{x})}{2R^{d+\alpha}}\left(\omega_d R^d-\left(\frac{2\|f\|_{L^p}}{f(\overline{x})}\right)^p\right)
= C_{\alpha,d}\frac{f(\overline{x})}{2R^{d+\alpha}} \left(\frac{2\|f\|_{L^p}}{f(\overline{x})}\right)^p\\
&=\frac{C_{\alpha,d} 2^p}{2\left(\frac{2}{\omega_d}\right)^{\frac{d+\alpha}{d}}2^{\frac{p(d+\alpha)}{d}}}
\frac{\|f\|_{L^p}^p f(\overline{x})^{\frac{p(d+\alpha)}{d}} f(\overline{x})}{(\|f\|_{L^p})^{\frac{p(d+\alpha)}{d}}f(\overline{x})^p}\\
&=C(\alpha,d,p)\frac{f(\overline{x})^{1+\frac{p\alpha}{d}}}{\|f\|_{L^p}^{\frac{p\alpha}{d}}}.
\end{aligned}
$$
If $R$ does not fulfils (\ref{eq:5.1}), then
$$
\left(\frac{2}{\omega_d}\right)^{\frac{1}{d}}\left(\frac{2\|f\|_{L^p}}{f(\overline{x})} \right)^{\frac{p}{d}}>M,
$$
thus, we conclude that
$$
f(\overline{x})\leq\frac{2}{M^{\frac{d}{p}}(\frac{\omega_d}{2})^{\frac{1}{p}}}\|f\|_{L^p}\leq C(d,p)\|f\|_{L^p}.
$$
This completes the proof of Lemma \ref{lem:5.1}.
\end{proof}

\begin{remark}
For the case of $\mathbb{R}^d$, we can refer to \cite{Rafael.2016}.
\end{remark}

\vskip .1in
\noindent \textbf{Acknowledgement}

\noindent The research are supported by the National Natural Science Foundation of China 11771284 and 11831011.
\vskip .3in
\bibliography{bib}

\end{document}